\documentclass[12pt]{article}

\usepackage{geometry}
\usepackage{booktabs}
\usepackage{amsfonts}
\usepackage{graphicx}
\usepackage{epstopdf}
\usepackage{mathtools}
\usepackage{pifont}
\usepackage{psfrag}

\ifpdf
  \DeclareGraphicsExtensions{.eps,.pdf,.png,.jpg}
\else
  \DeclareGraphicsExtensions{.pdf}
\fi

\allowdisplaybreaks

\usepackage{enumitem}
\usepackage{amssymb,amsthm}
\setitemize{label=\scriptsize{$\blacksquare$}}
\setenumerate{label=(\alph*)}

\DeclareMathOperator*{\argmin}{arg\,min}
\DeclareMathOperator*{\spa}{span}
\DeclareMathOperator*{\sinc}{sinc}
\DeclareMathOperator*{\range}{Ran}

\newcommand{\Ndet}{N_{\rm det}}
\newcommand{\Nt}{N_t}
\newcommand{\Nx}{N_x}
\newcommand{\Nrad}{N_r}

\usepackage{xspace}
\newcommand{\MATLAB}{\textsc{Matlab}\xspace}

\title{A Galerkin least squares approach for\\photoacoustic tomography\footnote{\textbf{Funding:}  Authors gratefully acknowledge the support of the Tyrolean Science Fund (TWF).}}

\author{Johannes Schwab\\Sergiy Pereverzyev Jr.\\
Markus Haltmeier\footnote{Corresponding author, {\tt markus.haltmeier@uibk.ac.at}.}}

\date{Department of Mathematics
University of Innsbruck\\
Technikerstrasse 13, A-6020 Innsbruck, Austria}

\usepackage{amsopn}

\newcommand{\edot}{\,\cdot\,}

\newcommand{\C}{\mathbb{C}}
\newcommand{\R}{\mathbb{R}}
\newcommand{\Z}{\mathbb{Z}}
\newcommand{\N}{\mathbb{N}}

\newcommand{\la}{\lambda}
\newcommand{\La}{\Lambda}
\newcommand{\eps}{\epsilon}
\newcommand{\trans}{\mathsf{T}}
\newcommand{\ball}{B_R(0)}
\newcommand{\xx}{x}

\newcommand{\Ao}{\mathbf A}
\newcommand{\Io}{\mathbf I}
\newcommand{\Bo}{\mathbf B}
\newcommand{\Mo}{\mathbf M}
\newcommand{\Wo}{\mathbf W}

\newcommand{\Po}{\mathbf P}
\newcommand{\Qo}{\mathbf Q}

\newcommand{\V}{\mathcal V}
\newcommand{\X}{\mathcal X}
\newcommand{\Y}{\mathcal Y}

\newcommand\rmd{\mathrm{d}}
\newcommand\ds{\rmd s}

\newcommand{\ga}{\gamma}

\newcommand{\ph}{\varphi}
\newcommand{\dd}{d}
\newcommand{\EE}{\mathcal{E}}
\newcommand{\RR}{\mathcal{R}}

\newcommand{\sph}{\mathbb S}

\newtheorem{theorem}{Theorem}
\newtheorem{lemma}[theorem]{Lemma}
\newtheorem{remark}[theorem]{Remark}

\makeatletter
\newcommand*\bigcdot{\mathpalette\bigcdot@{.6}}
\newcommand*\bigcdot@[2]{\mathbin{\vcenter{\hbox{\scalebox{#2}{$\m@th#1\bullet$}}}}}
\makeatother

\newcommand\ip[2]{\langle {#1},  {#2} \rangle}
\newcommand\inner[2]{{#1}\bigcdot {#2}}

\usepackage[color=red!60]{todonotes}

\numberwithin{equation}{section}
\numberwithin{figure}{section}
\numberwithin{table}{section}
\numberwithin{theorem}{section}

\newcommand{\bkl}[1]{\left(#1\right)}
\newcommand{\kl}[1]{(#1)}
\newcommand{\set}[1]{\{#1\}}
\newcommand{\abs}[1]{\lvert#1\rvert}
\newcommand{\norm}[1]{\lVert#1\rVert}

\usepackage{framed}
\begin{document}

\maketitle

\begin{abstract}
The development of fast and accurate image reconstruction algorithms is a central
aspect of computed tomography. In this paper we address this issue for
photoacoustic computed tomography in circular geometry. We investigate  the
Galerkin least squares method for that purpose. For approximating the function to be
recovered  we use subspaces of translation invariant spaces generated by a
single function.
This includes many systems that have  previously been  employed in PAT such
as generalized  Kaiser-Bessel basis functions or the natural pixel basis.
By exploiting an isometry property of the forward problem we are able to efficiently
set up the Galerkin equation  for a wide class of generating  functions  and devise
efficient   algorithms for its solution. We establish a convergence analysis and
present numerical simulations that demonstrate the efficiency and accuracy
of the derived algorithm.

\medskip
\noindent
\textbf{Key words:}
Photoacoustic imaging, computed tomography,  Galerkin least squares method, Kaiser-Bessel functions, Radon transform, least-squares approach.

\medskip
\noindent
\textbf{AMS subject classification:}
65R32, 45Q05, 92C55.
\end{abstract}

\section{Introduction}
\label{sec:intro}

Photoacoustic tomography (PAT) is an emerging non-invasive tomographic imaging
modality that allows high resolution imaging with high  contrast.
Applications  are ranging from breast screening in patients to whole body  imaging
of  small animals  \cite{beard2011biomedical,ntziachristos2005looking,KruKisReiKruMil03,wang2012photoacoustic}.
The basic principle of PAT is as follows. If a semitransparent sample is illuminated with a short pulse, then parts of the optical energy are absorbed inside the sample (see Figure~\ref{fig:pat}). This causes a rapid thermoelastic expansion, which in turns  induces an acoustic pressure wave. The pressure wave is measured outside  of the sample and used for reconstructing an image of the interior.

\begin{figure}[thb!]
\centering
  \includegraphics[width=0.8\textwidth]{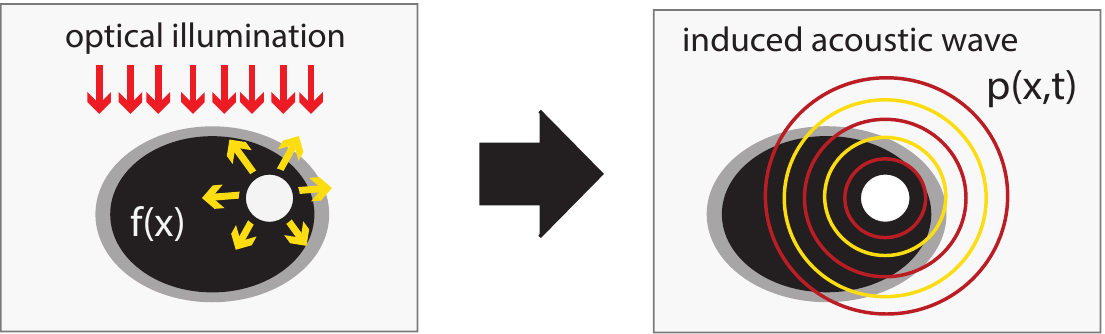}
\caption{\label{fig:pat}
\textsc{Basic principle of PAT.}  Pulsed optical illumination and  subsequent thermal expansion induces  an acoustic pressure wave. The pressure wave is measured outside of the object and used to obtain an image of the interior.}
\end{figure}

In this paper we work with the standard model of PAT, where the  acoustic pressure  $p \colon \R^\dd \times \kl{0, \infty} \to \R$
solves the  standard wave equation
\begin{equation} \label{eq:wave-fwd}
	\left\{ \begin{aligned}
	&\partial_t^2  p (x,t) - \Delta_x p(x,t)
	=
	0 \,,
	 && \text{ for }
	\kl{x,t} \in
	\R^\dd \times \kl{0, \infty} \,,
	\\
	&p\kl{x,0}
	=
	f(x)  \,,
	&& \text{ for }
	x  \in \R^\dd \,,
	\\
	&\partial_t
	p\kl{x,0}
	=0 \,,
	&& \text{ for }
	x  \in \R^\dd \,.
\end{aligned} \right.
\end{equation}
Here $d$ is the spatial dimension, $f \colon \R^\dd \to \R$ the absorbed  energy distribution, $\Delta_x$ the spatial Laplacian, and
$\partial_t$ the derivative with respect to the time variable $t$. The speed of sound is assumed to be constant and has been rescaled to
one. We further suppose that $f$ vanishes outside an open  ball $\ball \subseteq \R^\dd$.
The goal of PAT is to recover the function $f$ from measurements of $\Wo f \coloneqq p|_{ \partial \ball \times (0, \infty) }$.
Evaluation of $\Wo$  is referred to as  the direct problem
and the problem of reconstructing  $f$ from  (possibly approximate) knowledge
of $\Wo f $ as  the inverse problem of PAT. The cases $d=3$ and $d=2$ are of actual relevance in  PAT (see \cite{kuchment2011mathematics,BurBauGruHalPal07}).

In the recent years several solution methods for the inverse problem of PAT have been derived. These approaches can
be classified in direct methods on the one and iterative (model based) approaches on the other hand.
Direct methods are based on explicit solutions for inverting the operator $\Wo$ that can be implemented numerically.
This includes time reversal (see  \cite{burgholzer2007exact,Hristova2008,FinPatRak04,nguyen2016dissipative,Treeby10}), Fourier domain  algorithms (see \cite{AgrKuc07,HalSchBurNusPal07,Kun07b,palamodov2012uniform,salman14inversion,xu2002timedomain}),
and explicit reconstruction formulas  of the back-projection
type (see \cite{ansorg2013summability,FinHalRak07,FinPatRak04,haltmeier13inversion,haltmeier14universal,HalPer15a,HalPer15b,Kun07a,kunyansky2015inversion,natterer2012photo,nguyen2009family,xu2005universal}). Model based iterative approaches, on the other hand, are based on a discretization of the forward problem  together with numerical solution methods for solving the
resulting system of linear equations. Existing iterative approaches  use
interpolation based discretization (see \cite{DeaBueNtzRaz12,PalNusHalBur07b,PalViaPraJac02,RosNtzRaz13,zhang2009effects})
or approximation using radially symmetric basis functions  (see \cite{wang2014discrete,wang2012investigation}). Recently, also iterative schemes using
a continuous domain formulation of the adjoint have been studied, see \cite{arridge2016adjoint,belhachmi2016direct,haltmeier2016iterative}.
Direct methods are numerically efficient and robust and have similar
complexity as numerically evaluating the forward problem.
Iterative methods typically are slower  since the forward and adjoint problems have to be evaluated repeatedly. However, iterative methods have the advantage of being  flexible as one  can  easily add  regularization terms and incorporate measurement characteristics such as finite sampling, finite bandwidth and finite detectors size (see \cite{DeaBueNtzRaz12,haltmeier2010spatial,RoiEtAl14,wang2011imaging,wang2014discrete,XuWan03}). Additionally,  iterative methods tend to be more accurate
in the case of noisy data.

\subsection{Proposed Galerkin least squares approach}

In this paper we develop a  Galerkin approach for PAT that combines
advantages of  direct and  model based approaches. Our method  comes with a clear convergence theory, sharp error estimates and an efficient implementation.
The  Galerkin least squares method for $\Wo f = g$  consists in finding a minimizer of the
restricted least squares  functional,
\begin{equation} \label{eq:LN}
	f_N  \coloneqq \argmin \set{ \norm{\Wo h - g } \mid h  \in  \X_N }  \,,
\end{equation}
where $\X_N$ is a  finite dimensional reconstruction space   and  $\norm{\edot} $
an appropriate Hilbert space norm.
If $(\varphi^k_N)_{k \in\La_N}$ is a basis of $\X_N$  then
$f_N = \sum_{k \in\La_N} c_{N,k}\varphi^k_N$, where $c_N=(c_{N,k})_{k\in \La_N} $
is the unique solution
of the  Galerkin equation
\begin{equation}\label{eq:lsgal}
\Ao_N c_N = (\ip{\Wo\ph^k_N}{g})_{k \in\La_N}
\qquad \text{ with }  \; \Ao_N \coloneqq (\ip{\Wo \varphi^k_N}{\Wo \varphi_N^\ell})_{k,\ell \in \La_N}  \,.
\end{equation}
We call the matrix $\Ao_N$  the (discrete) imaging matrix.

In general, both the computation of the  imaging matrix as well as the  solution
of the Galerkin equation can be  numerically expensive.
In this paper we demonstrate that for the inverse problem of PAT,
both issues  can be efficiently implemented.
These  observations are based on the  following:
\begin{itemize}
\item \textsc{Isometry property.}
Using the isometry property of \cite{FinHalRak07,FinPatRak04}
one shows that the entries of the system matrix are given by
$\tfrac{R}{2}\ip{ \varphi^k_N}{ \varphi_N^\ell}_{L^2}$;
see~Theorem \ref{thm:leastsquares}.

\item \textsc{Shift invariance.} If, additionally, we take the basis functions
$\varphi^k_N$ as translates of a single
generating function $\varphi \in L^2(\R^\dd)$, then
$\ip{ \varphi^k_N}{ \varphi_N^\ell}_{L^2} = \ip{ \varphi^0_N}{ \varphi_N^{k-\ell}}_{L^2}$
for $k, \ell \in \La_N \subseteq \Z^\dd$ .
\end{itemize}
Consequently only $2^\dd \abs{\La_N}$ inner products have to be computed in our Galerkin
approach  opposed to   $\abs{\La_N}(\abs{\La_N}+1)/2 = \mathcal{O}(\abs{\La_N}^2)$ inner products required in the general case.
Further, the resulting shift invariant structure of the system matrix allows to
efficiently solve the Galerkin equation.

Note that  shift  invariant  spaces  are frequently employed in  computed tomography and include  splines spaces,   spaces of bandlimited functions, or spaces generated by
Kaiser-Bessel functions. In this paper we will especially use Kaiser-Bessel functions which are often considered as the most suitable basis for computed tomography       \cite{herman2015basis,lewitt1992alternatives,matej1996practical,nilchian2015optimized}. For the use in PAT they have first been proposed in~\cite{wang2014discrete}. We are not aware of existing Galerkin approaches for tomographic image reconstruction exploiting  isometry and shift invariance. However, we anticipate that similar methods can be derived for other tomographic problems, where an isometry property
is known (such as X-ray based CT~\cite{Kuc14,Nat01}).
We further note that our approach has close relations to the method of approximate
inverse, which has frequently been applied to computed tomography
\cite{HalSchuSch05,Lou96,LouMaa90,LouSchu96,rieder2000approximate,rieder2004approximate}. Instead of approximating the unknown function
using a prescribed reconstruction space,  the method of approximate inverse recovers prescribed moments of the unknown and is somehow dual to the    Galerkin approach.

\subsection{Outline}
The rest of this article is organized as follows. In Section~\ref{sec:galerkinPAT}
we apply the Galerkin least squares method for the inverse problem of PAT.
By using  the isometry property we derive a simple characterization of the
Galerkin equation in  Theorem~\ref{thm:leastsquares}.
We derive a convergence and stability result for the  Galerkin least squares
method applied to PAT (see Theorem~\ref{thm:convG}).
In Section~\ref{sec:spaces} we study shift invariant spaces for computed tomography.
As the main results in that section  we derive an estimate for  the
$L^2$-approximation  error using elements from the shift invariant space.
In Section~\ref{sec:galerkinshift}  we  present details for the  Galerkin approach
using subspaces of shift invariant  spaces.
In  Section \ref{sec:num} we present numerical studies using our
Galerkin approach and  compare it to related approaches in the literature.
The paper concludes with a conclusion and a short outlook in
Section~\ref{sec:conclusion}.

\section{Galerkin approach for PAT}
\label{sec:galerkinPAT}

Throughout the following, suppose $d\geq 2$, let   $\ball \coloneqq \set{\xx \in \R^\dd \mid  \norm{\xx} < R }$ denote the open ball with radius $R$ centered at the origin, and let
$ L^2_R(\R^\dd) \coloneqq \set{ f \in L^2\kl{\R^\dd} \mid  f(x) = 0 \text { for } \xx \in \R^\dd \setminus \ball}$ denote the Hilbert space of all square integrable functions which vanish outside $\ball$. For two measurable functions $ g_1,  g_2 \colon \partial \ball  \times (0, \infty) \to \R$ we write
 \begin{equation} \label{eq:tinner}
 	\ip{g_1}{g_2}_t \coloneqq
	\int_{\partial \ball }
	\int_0^\infty g_1(z,t)g_2(z,t)  \, t \, \rmd t \,  \ds(z)  \,,
\end{equation}
provided that the integral exists. We further denote by  $\Y$ the Hilbert space of
all functions $ g \colon \partial \ball  \times (0, \infty) \to \R$ with $\norm{g}_t^2 \coloneqq \ip{g}{g}_t < \infty$.

\subsection{PAT and the  wave equation}

For initial data $ f  \in C_c^{[d/2]+2}(\ball)$ consider
the wave equation  \eqref{eq:wave-fwd}.
The solution $p \colon \R^\dd \times (0, \infty) \to \R$  of
\eqref{eq:wave-fwd} restricted to the boundary of $\ball$  is denoted by
$\bar \Wo f \colon \partial \ball \times  (0, \infty) \to \R$.
The associated operator is defined by $\bar  \Wo \colon   C_c^{[d/2]+2}(\ball) \subseteq L^2_R(\R^\dd) \to \Y \colon f \mapsto \bar \Wo f$

\begin{lemma}[Isometry and continuous extension \label{lem:iso} of $\bar  \Wo$]\mbox{}
\begin{enumerate}
\item\label{lem:iso1} For all $f_1, f_2 \in  C_c^{[d/2]+2}(\ball)$ we have
$\ip{f_1}{f_2} = \tfrac{2}{R} \ip{\bar \Wo f_1}{\bar \Wo f_2}_t$.
\item\label{lem:iso2} $\bar \Wo$ uniquely extends to  a bounded linear operator $\Wo \colon L^2_R(\R^\dd) \to \Y$.
\item\label{lem:iso3}
For all $f_1, f_2  \in L^2_R(\R^\dd)$ we have $\ip{f_1}{f_2} = \tfrac{2}{R} \ip{\Wo f_1}{\Wo f_2}_t$.
\end{enumerate}
\end{lemma}

\begin{proof}\mbox{}
\ref{lem:iso1}: See~\cite[Equation (1.16)]{FinHalRak07} for $\dd$ even and
\cite[Equation (1.16)]{FinPatRak04} for $\dd$ odd.  (Note that
the isometry identities in \cite{FinHalRak07,FinPatRak04} are stated for  the wave equation with different initial conditions, and therefore at first glance look different from~\ref{lem:iso1}.)\\
\ref{lem:iso2}, \ref{lem:iso3}: Item~\ref{lem:iso1} implies that $\Wo$ is bounded with respect to the norms of $L^2_R(\R^\dd)$ and $\Y$ and defined on a dense subspace of $L^2_R(\R^\dd)$. Consequently it uniquely extends to a bounded operator $\Wo \colon L^2_R(\R^\dd) \to \Y$. The continuity of the inner product finally shows the isometry property on $L^2_R(\R^\dd)$.
\end{proof}

We call $\Wo$ the acoustic forward operator.
PAT is concerned with the inverse problem of estimating  $f$
from potentially noisy and  approximate knowledge of  $\Wo f$.
In this paper we use the Galerkin least squares method for that purpose.

\subsection{Application of the Galerkin method}

Let $(\X_N)_{N\in \N}$ and $(\Y_N)_{N\in \N}$  be families of subspaces
of $L^2_R(\R^\dd)$ and $\Y$, respectively, with $\dim \X_N = \dim \Y_N  < \infty$.  Further let  $\Qo_N $ denote the orthogonal projection on $\Y_N$ and suppose
$ g \in \Y$.  The  Galerkin method for solving $\Wo f = g$ defines the approximate solution  $f_N \in \X_N$ as the solution of
\begin{equation}\label{eq:galerkin}
	 \Qo_N \Wo f_N = \Qo_N g \,.
\end{equation}
In this paper we consider the  special case where $\Y_N =  \Wo \X_N$, in which case  the solution of \eqref{eq:galerkin} is referred to as \emph{Galerkin least squares method}.
The name comes from the fact that in this case the Galerkin solution can be uniquely characterized  as the minimizer of the
least squares functional over $\X_N$,
\begin{equation}\label{eq:leastsquares}
    \Phi_N(f_N) \coloneqq
    \frac{1}{2} \norm{ \Wo   f_N - g}_t^2
	\to \min_{f_N \in \X_N} \,.
\end{equation}
Because $\Phi_N$  is a quadratic functional on a finite dimensional space and $\Wo$ is injective, \eqref{eq:leastsquares}  possesses a unique solution.  Together with the  isometry property we obtain the following characterizations of the least squares Galerkin method for PAT.

\begin{theorem}[Characterizations of the Galerkin least squares method]\label{thm:leastsquares}
For $g \in \Y$ and $f_N \in  \X_N$
 the  following  are equivalent:
 \begin{enumerate}[label=(\arabic*)]
\item\label{thm:leastsquares1} $ \Qo_N \Wo f_N = \Qo_N g$;
\item\label{thm:leastsquares3} $f_N$ minimizes the least squares functional \eqref{eq:leastsquares};
\item\label{thm:leastsquares2}
For an arbitrary  basis $(\ph_N^k)_{k \in \La_N}$ of $\X_N$,
we have
\begin{equation} \label{eq:galerkinequation}
\Ao_N c_N = d_N
\end{equation}
where
\begin{itemize}
\item $c_N \coloneqq (c_{N,k})_k$ with $f_N  = \sum_{k\in \La_N} c_{N,k} \ph_N^k$;
\item $d_N \coloneqq (\ip{\Wo \ph_N^k}{g}_t)_{k\in \La_N}$;
\item $\Ao_N \coloneqq ( \tfrac{R}{2}\ip{\ph_N^k}{\ph_N^\ell}_{L^2})_{k,\ell\in \La_N}$.
\end{itemize}
\end{enumerate}
\end{theorem}

\begin{proof}
The equivalence  of \ref{thm:leastsquares1}  and \ref{thm:leastsquares3} is
a standard result for the Galerkin squares method  (see, for example,  \cite{Kre99}).
 Another  standard  characterization   shows the  equivalence of \ref{thm:leastsquares1}  and  \ref{thm:leastsquares2} with the system matrix $\Ao_N =( \ip{\Wo \ph_N^k}{\Wo \ph_N^\ell}_t)_{k,\ell \in \La_N}$. Now, the isometry property given in Lemma \ref{lem:iso} shows $\ip{\Wo\ph_N^k}{\Wo\ph_N^\ell}_t
=\tfrac{R}{2}\ip{\ph_N^k}{\ph_N^\ell}_{L^2})_{k,\ell \in \La_N}$ and concludes the proof. \end{proof}

In general, evaluating all matrix entries $\ip{\Wo \ph_N^k}{\Wo \ph_N^\ell}_t $ can be difficult.
For many basis functions an explicit  expression for $\Wo \ph_N^k$ is not available including the natural pixel basis,  spaces defined by linear interpolation, or spline spaces.
Hence $\Wo \ph_N^k$ has to be evaluated numerically which is
time consuming and introduces additional errors.  Even if
$\Wo \ph_N^k$  is given explicitly, then the inner products $\ip{\Wo \ph_N^k}{\Wo \ph_N^\ell}_{L^2} $ have to be computed numerically and stored.
For large $N$ this can be problematic and time consuming. In contrast, by using the isometry property in our approach we only
have to compute the inner products $\ip{\ph_N^k}{\ph_N^\ell}$.
Further, in computed tomography it is common to take $\ph_N^k$ as translates of a  single function $\ph_N^0$. In such a situation the inner products  satisfy $\ip{\ph_N^k}{\ph_N^\ell} = \ip{\ph_N^0}{\ph_N^{\ell-k}}$ and therefore only a small fraction of all inner products actually have to be  computed.

\subsection{Convergence and stability analysis}

As another consequence of the isometry property we derive linear error estimates
for the Galerkin approach to PAT. We consider noisy data where  the data $g^\delta \in \Y$
is known to satisfy
 \begin{equation} \label{eq:noise}
 \| \Wo f^0 -  g^\delta \| \leq \delta \,,
 \end{equation}
for some noise level $\delta \geq 0$ and unknown $f^0 \in L^2_R(\R^\dd)$.  For noisy data we define the
Galerkin least squares solution by
\begin{equation}\label{eq:leastsquares-delta}
	f_N^\delta  =
	\argmin \left\{ \norm{ \Wo   h - g^\delta}_t
	\mid h \in \X_N \right\} \,.
\end{equation}
We then have the following convergence and stability result.

\begin{theorem}[Convergence and stability of the Galerkin method for PAT] \label{thm:convG}
 Let $f^0 \in L^2_R(\R^\dd)$, $g^\delta \in \Y$, $\delta \geq 0$
 satisfy \eqref{eq:noise} and let $f_N^\delta$ be defined by
 \eqref{eq:leastsquares-delta}. Then, the following error estimate for the
 Galerkin method holds:
 \begin{equation} \label{eq:galerkin-est}
 \norm{f_N^\delta - f^0  } \leq
  \min \set{\norm{h-f^0} \mid h \in \X_N}
 +
 \sqrt{\frac{2}{R}}  \, \delta  \,.
 \end{equation}
 \end{theorem}

\begin{proof}
We start with the noise free case  $\delta =0$. The definition of $f_N^\delta$  and the isometry property of $\Wo$ yield
\begin{align*}
f_N^0
&=  \argmin \left\{ \norm{ \Wo   h - g^0 }_t
	\mid h \in \X_N \right\}
	\\
&=  \argmin \left\{ \norm{ \Wo   h - \Wo f^0 }_t
	\mid h \in \X_N \right\} \\
&=  \argmin \left\{\norm{   h -  f^0 }
	\mid h \in \X_N \right\}
\end{align*}
This shows $f_N^0  = \Po_{\X_N} f^0$ and yields  \eqref{eq:galerkin-est} for $\delta =0$.
Here and below we use $\Po_V$  to denote the orthogonal projection on
a closed subspace  $V \subseteq  L^2_R(\R^\dd)$.

Now consider the case of arbitrary $\delta$,  with $ g^\delta  = \Wo f^0
+ e^\delta$ where   $e^\delta \in \Y$ satisfies  $\norm{e^\delta} \leq \delta$.
Because  $\range (\Wo)$ is closed we can
write
\begin{equation*}
g^\delta =
( \Wo f^0+ \Po_{\range(\Wo)} (e^\delta) ) + \Po_{\range(\Wo)^\bot} (e^\delta)
\eqqcolon \Wo f^\delta + \Po_{\range(\Wo)^\bot} (e^\delta) \,.
\end{equation*}
Following the  case   $\delta=0$ and using that  $\Po_{\range(\Wo)^\bot} (e^\delta) \bot \range(\Wo)$ one verifies that $f_N^\delta  = \Po_{\X_N} f^\delta$.
Therefore, by the triangle inequality and the isometry property of $\Wo$
we obtain
\begin{align*}
\norm{f_N^\delta - f^0}
&\leq
\norm{f_N^\delta - f_N^0}
+ \norm{ f_N^0 - f^0} \\
&  =
\norm{\Po_{\X_n} (f^\delta - f^0)}
+  \min \set{\norm{h-f^0} \mid h \in \X_N}
\\
&\leq \sqrt{\frac{2}{R}} \,
\norm{\Wo f^\delta - \Wo f^0 }_t
+  \min \set{\norm{h-f^0} \mid h \in \X_N}  \,.
\end{align*}
Together with $\norm{\Wo f^\delta - \Wo f^0}_t  = \norm{\Po_{\range(\Wo)} (e^\delta) }_t  \leq \delta$ this concludes the proof.
\end{proof}

The error estimate  in Theorem~\ref{thm:convG}  depends on two terms: the first term
depends on the approximation properties of the space $\X_N$  and the second term on the noise level $\delta$.   As easily verified   both terms are  optimal and cannot  be improved.  The second term shows stability of our  Galerkin least squares approach. Under the reasonable  assumption that the spaces $\X_N$ satisfy the  denseness property
\begin{equation*}
	\forall f \in L^2_R(\R^\dd) \colon \quad
	\lim_{N \to \infty} \min \set{\norm{h-f} \mid h \in \X_N} = 0 \,,
 \end{equation*}
the derived error estimate further implies  convergence of the
Galerkin approach.

\section{Shift invariant spaces in computed tomography}
\label{sec:spaces}

In many tomographic and signal processing applications, natural spaces for approximating the
underlying function  are subspaces of shift invariant spaces. In this paper we  consider spaces $\V_{T,s,\ph} $
that are generated by  translated and scaled versions of a single function $\ph\in L^2(\R^\dd)$,
\begin{equation} \label{eq:XTs}
	\V_{T,s,\ph} \coloneqq
	\overline{\spa (\set{  \ph_{T,s}^k  \mid k \in \Z^\dd  }})
	\subseteq L^2(\R^\dd) \,.
\end{equation}
Here $\spa$ denotes the linear hull, $\overline{X}$ stands for the closure  with respect  to    $ \norm{\edot}_{L^2}$ of a set $X$, and
\begin{equation} \label{eq:phik}
\ph_{T,s}^k(x) \coloneqq  \frac{1}{T^{\dd/2}} \,
\ph \left( \frac{x}{T}-sk \right) \quad \text{  for $T,s >0$ and $k \in \Z^\dd$ } \,.
\end{equation}
 We have chosen the  normalization of the generating
 functions $\ph_{T,s}^k$ in such a way that   $\norm{\ph_{T,s}^k}_{L^2} = \norm{\ph}_{L^2}$ for all $T,s,k$. In this section we derive  conditions such that any $L^2$ function can be approximated
by elements in $\V_{T,s,\ph}$. Further, we present examples  of generating functions that
are relevant for (photoacoustic) computed tomography.

Any tomographic reconstruction method uses, either explicitly or implicitly,
a particular discrete reconstruction space.
This is obvious for any iterative procedure as it requires a finite dimensional representation of the forward operator that can be evaluated numerically. However, also direct methods use an underlying discrete image space. For example,
standard filtered backprojection algorithms  usually reconstruct samples of a bandlimited approximation of the unknown function.
In such a situation, the underlying discrete signal space consists of bandlimited functions. In this paper we allow more general shift invariant spaces.

The following properties of the generating function and the spaces $\V_{T,s,\ph}$
have been reported desirable for tomographic applications (see \cite{nilchian2015optimized,wang2014discrete}):
\begin{enumerate}[label=(V\arabic*)]
\item\label{V1} $\ph$ has ``small'' spatial support;
\item\label{V2} $\ph$ is rotationally invariant;
\item\label{V3} $(\ph_{T,s}^k)_{k \in \Z^\dd}$ is a Riesz basis of $\V_{T,s,\ph}$;
\item\label{V4} $\ph$ satisfies the so called partition  of unity property.
\end{enumerate}
Conditions \ref{V1} and \ref{V2} are desirable from a computational point
of view  and often help to derive  efficient reconstruction algorithms.
The properties \ref{V3} and \ref{V4} are of more fundamental nature as these conditions
imply that any $L^2$ function can be approximated arbitrarily well by  elements in
$\V_{T,s,\ph}$ as $T \to 0$  (with $s$ kept fixed; the so  called stationary case).
In \cite{nilchian2015optimized} it has been pointed out  that the properties
\ref{V1}-\ref{V4} cannot be  simultaneously fulfilled. This implies that
for taking $s$ independent of $T$, the spaces $\V_{T,s,\ph}$ have a
limited approximation capability in the sense that for a typical function $f$,
the approximation error $\min_{u \in \V_{T,s,\ph}} \norm{f-u}^2_{L^2}$
does not converge to zero as $T\to 0$ and $s$ is kept fixed.

Despite these negative results, radially symmetric basis functions are of great popularity in computed tomography (see for example, \cite{herman2015basis,lewitt1990multidimensional,lewitt1992alternatives,matej1996practical,nilchian2015optimized,wang2012investigation,wang2014discrete}).  In this paper we therefore propose to also allow the shift parameter $s$ to be variable.
Under reasonable assumptions we show that  the approximation error  converges to zero for
 $s \to 0$. This convergence in particularly holds for radially symmetric generating functions having  some
decay in the Fourier space, including generalized Kaiser-Bessel functions which are the most
popular choice in tomographic image reconstruction.

\subsection{Riesz bases of shift invariant spaces}

Recall that  the family $(\ph_{T,s}^k)_{k \in \Z^\dd}$ is called a Riesz
basis of $\V_{T,s,\ph}$ if there exist $A,B>0$ such that
\begin{equation} \label{eq:rb}
\forall c \in \ell^2(\Z^\dd)\colon \quad
A\norm{c}_{\ell^2}^2\leq
\Bigl\lVert\sum_{k\in\Z^\dd} c_k \ph_{T,s}^k\Bigr\rVert_{L^2}^2\leq B\norm{c}_{\ell^2}^2 \,,
\end{equation}
where $\norm{c}_{\ell^2}^2  \coloneqq \sum_{k\in\Z^\dd} \abs{c_k}^2$ is the squared $\ell^2$-norm of  $c=(c_k)_{k \in \Z^\dd}$.  A Riesz basis  of $\V_{T,s,\ph}$ can equivalently be defined as a linear independent family of frames  and the constants $A$ and $B$ are the
lower and upper frame bounds of  $(\ph_{T,s}^k)_{k \in \Z^\dd}$, respectively.
In the following we write $\hat \ph$ for the $d$-dimensional  Fourier transform defined by $\hat \ph (\xi) \coloneqq (2\pi)^{-d/2} \int_{\R^\dd}
\ph (x) e^{-i \inner{\xi}{x} } \rmd x$ for $\ph \in L^2(\R^\dd)\cap L^1(\R^\dd)$ and extended to $L^2(\R^\dd)$ by continuity.

The following two Lemmas are well known in the case that
$d = T=1$  (see~\cite[Theorem 3.4]{Mal09}). Due to page limitations
and because the general case is shown analogously, the  proofs of the Lemmas are omitted.

\begin{lemma}[Riesz basis property] \label{lem:riesz}
The family $(\ph_{T,s}^k)_{k\in\Z^\dd}$ is a Riesz basis of $\V_{T,s,\ph}$
with frame bounds $A$ and $B$, if and only if
\begin{equation}\label{eq:riesz}
\frac{A}{(2\pi)^\dd}\leq \frac{1}{s^\dd}\sum_{k\in\Z^\dd}|\hat{\ph}(\xi+\tfrac{2\pi}{s}k)|^2\leq \frac{B}{(2\pi)^\dd}
\quad \text{ for a.e. }  \;  \xi\in[0,\tfrac{2\pi}{s}]^\dd \,.
\end{equation}
\end{lemma}

\begin{proof}
Follows the lines of  \cite[Theorem 3.4]{Mal09}.
\end{proof}

The following  Lemma implies that for any Riesz basis $(\ph_{T,s}^k)_{k\in\Z^\dd}$ one  can construct an orthonormal basis of $\V_{T,s,\ph}$ that is again generated by translated and scaled versions $\theta_{T,s}^k(x) \coloneqq  T^{-\dd/2} \theta (x/T - sk)$
of a single function $\theta \in L^2(\R^\dd)$.

\begin{lemma}[Orthonormalization]\label{lem:on}
Let $(\ph_{T,s}^k)_{k\in\Z^\dd}$ be a Riesz basis of $\V_{T,s,\ph}$.
\begin{enumerate}
\item \label{lem:on1}$(\ph_{T,s}^k)_{k\in\Z^\dd}$  orthonormal $\iff$
$
\sum_{k\in\Z^\dd} |\hat{\ph}(\xi+\tfrac{2\pi}{s}k)|^2=\frac{s^\dd}{(2\pi)^\dd}
$ for a.e. $\xi  \in \R^\dd$.
\item \label{lem:on2}
$(\theta_{T,s}^k)_{k\in\Z^\dd}$ is an orthonormal basis of $\V_{T,s,\ph}$, where $\theta \in L^2(\R^\dd)$ is defined by
\begin{equation}\label{eq:on}
	\hat{\theta}(\xi)
	=
	\frac{s^{\dd/2} \hat{\varphi}(\xi)}{(2\pi)^{\dd/2} \sqrt{\sum_{k\in\Z^\dd}|\hat{\varphi}(\xi+\tfrac{2\pi}{s}k)|^2}}.
\end{equation}
\end{enumerate}
\end{lemma}

\begin{proof}
Follows the lines of  \cite[Theorem 3.4]{Mal09}.
\end{proof}

According to Lemma~\ref{lem:on}, for  theoretical purposes one may assume  that the considered basis of $\V_{T,s,\ph}$ is already orthogonal. From a practical point of view, however, it may be more convenient to work with the original non-orthogonal basis. The function $\ph$ may have additional properties such as small support or radial symmetry which may  not be the case for $\theta$.
Also it may not be the case that $\theta$ is known analytically.

\subsection{The $L^2$-approximation error}

We  now investigate the  $L^2$-approximation error in shift invariant spaces,
\begin{equation}\label{eq:aerror-def}
    \forall f \in L^2(\R^\dd) \colon
    \quad
    \min_{u \in \V_{T,s,\ph}}  \norm{f - u}_{L^2}
    = \norm{f - \Po_{T,s} f}_{L^2} \,,
\end{equation}
as well as its asymptotic properties.  Here and in  the following
$\Po_{T,s}\colon L^2(\R^\dd)\rightarrow \V_{T,s,\ph}$  denotes the  orthogonal projection on $\V_{T,s,\ph}$.
It is given by $\Po_{T,s} f = \sum_{\la \in \La}\langle f,e_\la\rangle e_\la$, where $\kl{e_\la}_{\la\in\La}$ is any orthogonal basis of  $\V_{T,s,\ph}$.
For the stationary case $s=1$, the following Theorem has been obtained in~\cite{blu99approximation}.

\begin{theorem}[The $L^2$-approximation error]
Let\label{thm:aerror}  $\kl{\ph_{T,s}^k}_{k\in\Z^\dd}$ be a Riesz basis of $\V_{T,s,\ph}$ and define
\begin{equation}\label{eq:RR}
\EE_{\ph}(s,T\xi)
 \coloneqq
 1-\frac{|\hat{\varphi}(T\xi)|^2}{\sum_{k\in\Z^\dd}|\hat{\varphi}(T\xi+2k\pi/s)|^2}
 \quad \text{ for } \xi \in \R^\dd  \text{ and } T,s> 0 \,.
\end{equation}
Then, for every $f \in W_2^r(\R^\dd)$ with $r>d/2$ we have
\begin{equation} \label{eq:aerror}
\norm{\Po_{T,s}f-f}_{L^2}
=\left[\int_{[-\frac{\pi}{Ts},\frac{\pi}{Ts}]^\dd}\abs{\hat{f}(\xi)}^2
\EE_{\ph}(s,T\xi)  \rmd \xi \right]^{\frac{1}{2}}
+\RR_{\ph}(f,Ts) \,,
\end{equation}
where the remainder can be estimated as
\begin{equation}
\label{eq:EE}
\RR_{\ph}(f,Ts)
\leq  \norm{f}_{W^r_2}
\bkl{\frac{Ts}{\pi}}^{r}
\sqrt{\sum_{n\in\Z^\dd\setminus\{0\}} \frac{1}{\lVert n \rVert^{2r}}}
\quad \text{ for }  T, s > 0\,.
\end{equation}
\end{theorem}

\begin{proof}
Let $(\theta_{T,s}^k)_{k\in\Z^\dd}$ denote the orthonormal basis of the space $\V_{T,s,\ph}$ as constructed in Lemma~\ref{lem:on}. Further,
for every $n\in \Z^\dd$ define  $Q_n\coloneqq\frac{2\pi}{Ts}n+[-\frac{\pi}{Ts},\frac{\pi}{Ts}]^\dd$
and  define functions $f_n \in L^2(\R^\dd)$  by its Fourier representation
\begin{equation*}
\hat{f}_n(\xi)=
\begin{cases}\hat{f}(\xi) \quad &\text{ if }\xi \in Q_n \,,\\
0 &\text{ if }\xi \not\in Q_n \,.
\end{cases}
\end{equation*}
Then  we have $f =\sum_{n\in\Z^\dd} f_n$ and
$\Po_{T,s}f - f =\sum_{n\in\Z^\dd} \Po_{T,s}f_n-f_n$ .

Now for every $n\in\Z^\dd$, we investigate the  approximation error
$\|\Po_{T,s}f_n-f_n\|^2$. We have
$\|\Po_{T,s}f_n-f_n\|^2 = \|f_n\|^2-\sum_{k\in\Z^\dd}|\langle f_n,\theta_{T,s}^k\rangle|^2$.
Further,
\begin{align*}
\langle f_n,\theta_{T,s}^k\rangle &= \langle \hat{f}_n,\hat{\theta}_{T,s}^k\rangle\\
&=T^{d/2}\int_{Q_n}\hat{f}_n(\xi)   \overline{\hat{\theta}(T\xi) e^{-iTs \inner{\xi}{k}}} \rmd \xi\\
&=T^{d/2}\int_{[-\frac{\pi}{Ts},\frac{\pi}{Ts}]^\dd}
\hat{f}_n(\xi- 2\pi n/(Ts))  \overline{ \hat{\theta}(T(\xi-  2\pi n/(Ts)))} e^{ iTs\inner{(\xi- 2\pi n/(Ts))}{k}}  \rmd \xi\\
&= T^{d/2}\int_{[-\frac{\pi}{Ts},\frac{\pi}{Ts}]^\dd}
\hat{f}_n(\xi- 2\pi n/(Ts)) \overline{ \hat{\theta}(T(\xi-2\pi n/(Ts)))} e^{ iTs\inner{\xi}{k}}  \rmd \xi \\
&=T^{d/2}  \hat{d}_{n,-k} 
\,,
\end{align*}
where $\hat{d}_{n,k}$ is the $k$-th Fourier-coefficient of the $2\pi/(Ts)$-periodization of the function $\xi\mapsto\hat{f}_n(\xi- 2\pi n/(Ts)) \,   \overline{ \hat{\theta}(T(\xi-2\pi n/(Ts)))}$. Due to Parseval's identity  we have
\begin{align*}
\sum_{k\in\Z^\dd}|&\langle f_n,\theta_{T,s}^k\rangle|^2\\
&=T^\dd\sum_{k\in\Z^\dd} \abs{\hat{d}_{n,k}}^2\\
&=T^\dd\frac{(2\pi)^\dd}{(sT)^\dd}\int_{Q_n}|\hat{f}_n(\xi)|^2|\hat{\theta}(T\xi)|^2\rmd\xi\\
&=\int_{Q_n}|\hat{f}_n(\xi)|^2\frac{|\hat{\varphi}(T\xi)|^2}{\sum_{k\in\Z^\dd}|\hat{\varphi}(T\xi+2k\pi/s)|^2} \rmd\xi.
\end{align*}
Therefore we obtain
\begin{align*}
\|\Po_{T,s}f_n-f_n\|^2&=\int_{Q_n}|\hat{f}_n(\xi)|^2\left(1-\frac{|\hat{\varphi}(T\xi)|^2}{\sum_{k\in\Z^\dd}|\hat{\varphi}(T\xi+2k\pi/s)|^2}\right)\rmd\xi\\
&=\int_{Q_n}|\hat{f}_n(\xi)|^2 \EE_{\ph}(s,T\xi) \rmd \xi \,.
\end{align*}

Next notice that for $ n\in\Z^\dd\setminus\{0\}$ and $\xi\in Q_n\subseteq \R^\dd$ we have $\|\xi\|\geq\frac{\pi}{Ts} \, \norm{n}$.
Therefore we can estimate
\begin{equation*}
\norm{\Po_{T,s}f_n-f_n}
\leq
\bkl{ \frac{Ts}{\pi }}^r
\bkl{ \frac{1}{\norm{n}} }^r
\bkl{  \int_{Q_n}\|\xi\|^{2r}|\hat{f}_n(\xi)|^2 \EE_{\ph}(s,T\xi) \rmd\xi }^{\frac{1}{2}} \,.
\end{equation*}
Together with the triangle inequality and the Cauchy-Schwarz inequality for sums
we  obtain
\begin{align*}
\|&\Po_{T,s}f-f\|
\\
&\leq\sum_{n\in\Z^\dd}\|\Po_{T,s}f_n - f_n\|
\\
&\leq  \left(\int_{Q_0}|\hat{f}(\xi)|^2 \EE_{\ph}(s,T\xi) \rmd \xi\right)^{\frac{1}{2}}
+
\left(\frac{Ts}{\pi}\right)^r \sum_{n\neq 0}\frac{1}{\|n\|^{r}}
\left(\int_{Q_n}\|\xi\|^{2r} |\hat{f}_n(\xi)|^2 \rmd \xi \right)^{\frac{1}{2}}
\\
&\leq \left(\int_{Q_0}|\hat{f}(\xi)|^2 \EE_{\ph}(s,T\xi) \rmd \xi\right)^{\frac{1}{2}}
+
\bkl{ \frac{Ts}{\pi}}^r
\bkl{ \sum_{n\neq 0}\frac{1}{\|n\|^{2r}}}^{\frac{1}{2}}
\bkl{ \int_{\R^\dd\setminus Q_0}\|\xi\|^{2r} |\hat{f}(\xi)|^2 \rmd\xi }^{\frac{1}{2}}
\\
\\
&\leq \left(\int_{Q_0}|\hat{f}(\xi)|^2 \EE_{\ph}(s,T\xi) \rmd\xi\right)^{\frac{1}{2}}
+
\bkl{ \frac{Ts}{\pi}}^r
\bkl{ \sum_{n\neq 0}\frac{1}{\|n\|^{2r}}}^{\frac{1}{2}}
\norm{f}_{W^r_2}
\,.
\end{align*}
Here the sum  $\sum_{n\neq 0} \|n\|^{-2r}$ is convergent because $r >d/2$.
After recalling that $Q_0 = [-\pi/(Ts),\pi/(Ts)]^\dd$, the above  estimate
yields~\eqref{eq:aerror}.
\end{proof}

Note that the remainder in Theorem~\ref{thm:aerror} satisfies
$\RR_{\ph}(f,Ts) \to 0$ as $Ts  \to 0 $. Consequently, for every
 sequence $(T_N, s_N)_{N\in \N}$ we have  $\lim_{N \to \infty} \norm{\Po_{T_N,s_N} f - f}_{L^2}^2 =0$  if
$T_N s_N  \to 0 $ and
\begin{equation*}
\int_{[-\frac{\pi}{T_Ns_N},\frac{\pi}{T_Ns_N}]^\dd}\abs{\hat{f}(\xi)}^2
\underbrace{\bkl{ 1-  \frac{\abs{\hat{\varphi}(T_N\xi)}^2}{\sum_{k\in\Z^\dd}\abs{\hat{\varphi}(T_N\xi+2k\pi/s_N)}^2}}}_{= \EE_{\ph}(s_N,T_N \xi )}   \rmd \xi \to 0  \,.
\end{equation*}
By   Lebesgue's dominated convergence theorem this holds if
$\EE_{\ph}(s_N, T_N \xi )$ almost everywhere  converges to $0$
as $N \to \infty$.
In the following theorem we consider  two possible sequences where
this is the case. Note that  Item~\ref{thm:conv1} in  that  theorem is
well known (see, for example, \cite{Mal09}), while Item~\ref{thm:conv2}  to the best of our knowledge is new.

\begin{theorem}[Asymptotic  behavior of $\EE_{\ph}$]\label{thm:conv}
Let $\varphi\in  L^2(\R^\dd) \cap L^1(\R^\dd)$.
\begin{enumerate}
\item\label{thm:conv1}
Suppose that $\hat{\varphi}(0) > 0$.
Then, for  every $s \in(0,\infty)$ we have that
$\lim_{T\rightarrow 0} \EE_{\ph}(s,T \xi ) = 0$ almost everywhere
if and only if
\begin{align} \label{eq:pou}
\frac{1}{\hat\ph(0)}
\sum_{m\in\Z^\dd}\varphi(x-ms) =  \frac{(2\pi)^{\dd/2}}{s^\dd} 
\quad \text{  for almost every  $x \in \R^\dd$} \,.
 \end{align}
Equation \eqref{eq:pou} is called the partition of unity property.

\item\label{thm:conv2} Suppose
$\hat{\varphi}(\xi)=\mathcal{O}(\|\xi\|^{-p})$ as $\|\xi\| \rightarrow \infty$ for some $p>d/2$. Let $(T_N)_{N\in \N}$ and
$(s_N)_{N\in \N}$ be bounded sequences in $(0, \infty)$
with $s_N \to 0$ as  $N \to \infty$. Then
\begin{equation} \label{eq:conv2}
\lim_{N \to \infty }
\EE_{\ph}(s_N, T_N \xi )
=0
\quad \text{  for  every  $\xi \in \R^\dd$}\,.
\end{equation}
\end{enumerate}
\end{theorem}

\begin{proof}\mbox{}
\ref{thm:conv1}
We have
\begin{align*}
&\lim_{T\to 0} \EE_{\ph}(s,T \xi )  = 0
 \text{ for a.e.  $\xi \in \R^\dd$ } \\
\iff
&\lim_{T\rightarrow 0}\sum_{k\neq 0}|\hat{\varphi}(T\xi+ 2k\pi/s)|^2=0
\text{ for a.e.  $\xi \in \R^\dd$ }
\\
\iff& \forall k \in \Z^\dd\setminus\set{0}  \colon
\hat{\varphi}( 2k\pi/s)=0\\
\iff&  \forall k \in \Z^\dd\setminus\set{0}  \colon
\int_{\R^\dd}\varphi(x) e^{i 2\pi \inner{x}{k}/s} \rmd x =0\\
\iff&\forall k \in \Z^\dd\setminus\set{0} \colon
\int_{[0,s]^\dd}\sum_{m\in\Z^\dd}\varphi(x-ms)e^{i 2\pi \inner{x}{k}/s} \rmd x=0\\
\iff& \sum_{m\in\Z^\dd}\varphi(x-ms)=   \frac{(2\pi)^{\dd/2}}{s^\dd} \, \hat\ph(0)  
\text{ for a.e.  $x \in \R^\dd$\,. }
\end{align*}

\ref{thm:conv2}
As  $\hat{\varphi}(\xi)=\mathcal{O}(\|\xi\|^{-p})$ for $\|\xi\|\rightarrow \infty$ there exist constants $R, C>0$ such that for  $\|\xi\|>R$ we have $ \abs{\hat{\varphi}(\xi)}\leq C \|\xi\|^{-p}$. Further, for all $\xi \in \R^\dd$
and $k\neq 0$ we have $ \|T_N\xi- 2\pi k/s_N\| \to \infty$. Therefore it exists $N_0\in \N$, such that for all $N\geq N_0$ we have
$\|T_N\xi-2\pi k / s_N\|>c$  and $\|T_N\xi\|\leq\frac{1}{2}\|2\pi
k/s_N\|$ for $k\neq 0$. Therefore, for all $N\geq N_0$,
\begin{align*}
\sum_{k\neq 0}|\hat{\varphi}(T_N\xi-\frac{2\pi}{s_N}k)|^{2}&\leq
C \sum_{k\neq 0} \|T_N\xi-\frac{2\pi}{s_N}k\|^{-2p}\\
&\leq C\sum_{k\neq 0} \left| \norm{\frac{2\pi}{s_N}k}-\|T_N\xi\|\right|^{-2p}\\
&\leq
C\sum_{k\neq 0} \left|\norm{\frac{2\pi}{s_N}k}-\frac{1}{2}\norm{\frac{2\pi}{s_N}k}\right|^{-2p}\\
&\leq C \left(\frac{s_N}{\pi}\right)^{2p}
\sum_{k\neq 0}\|k\|^{-2p} \,,
\end{align*}
which implies \eqref{eq:conv2}.
Note that $\sum_{k\neq 0}\|k\|^{-2p}$  is convergent because
$p >d/2$.
\end{proof}

From Theorems~\ref{thm:aerror} and~\ref{thm:conv} one concludes that
the system of $(\ph_{T,s}^k)_{k\in \Z^\dd}$ yields a vanishing  approximation error $\min_{u \in \V_{T,s,\ph}} \norm{f-u}^2_{L^2}$
in either of the  following  cases:
\begin{enumerate}
\item $\ph$ satisfies the partition of unity property,
$s$ is fixed and  $T \to 0$;
\item
$\hat \ph(\xi) =  \mathcal O(\norm{\xi}^{-d/2-\eps})$ for $\norm{\xi}\to \infty$, $T$ is  bounded and $s \to 0$.
\end{enumerate}

In both cases one could derive quantitative error estimates.
We do not investigate this issue further since our main
emphasis is  pointing out that allowing  $s$ to vary yields asymptotically vanishing approximation
error without the partition of unity property. This is relevant since the partition of unity property cannot be satisfied by any radially symmetric compactly supported
function.

Below we study two basic examples for generating functions
where Theorems~\ref{thm:aerror} and~\ref{thm:conv} can be applied.
These are pixel (or voxel) basis functions and
generalized Kaiser-Bessel functions. We focus on these
basis functions  since the pixel basis has been the most common choice in early tomographic image reconstruction while generalized Kaiser-Bessel functions are
currently considered as the method of choice. Further, also some standard   finite  element  bases satisfy the partition of unit property; compare
 with Remark~\ref{rem:FE}.

\subsection{Example: The pixel basis}
\label{sec:ex:pixel}

The pixel basis (also called  voxel basis in the case $d  > 2$)  has been
frequently used for image representation in  early tomographic image reconstruction  (see, for example \cite{gordon1970algebraic,herman2015basis,kak2001principles}).
It consists of  scaled and translated version of
the indicator function of the hyper-cube $[-1/2, 1/2[^\dd$
\begin{equation} \label{eq:chid}
\chi  \colon \R^\dd \to \R \colon
x \mapsto
\begin{cases}
1 & \text{ if } x \in [-1/2, 1/2[^\dd \\
0 & \text{ otherwise } \,.
\end{cases}
\end{equation}
For every $T, s>0$, the  family $(\chi_{T,s}^k)_{k\in \Z^\dd}$
with  $\chi_{T,s}^k (x)  = T^{-d/2}  \chi ((x- Ts k)/ T)$ clearly forms a
Riesz basis of
\begin{equation*}
\V_{T,s,\chi} =
\overline{\spa\set{\chi_{T,s}^k \mid k \in \Z^\dd} } \,.
\end{equation*}
 Note that the Fourier transform  of $\chi$ is given by
\begin{equation} \label{eq:chidhat}
\hat \chi  \colon \R \to \C \colon
\xi \mapsto (2\pi)^{-d/2} \sinc \bkl{\frac{\xi}{2}}
\coloneqq
(2\pi)^{-d/2} \prod_{j=1}^\dd \sinc \bkl{\frac{\xi_j}{2}} \,,
\end{equation}
where $\sinc(a) \coloneqq \sin(a)/a$ for $a\neq 0$ and
$\sinc(0) \coloneqq 1$.
We see $\hat \chi (\xi)= \mathcal O(\norm{\xi}^{-1})$ as $\norm{\xi} \to \infty$.
Consequently,  we cannot conclude from Theorem~\ref{thm:aerror} that
the spaces $\V_{T,s,\chi}$ yields an asymptotically
vanishing approximation error for $s \to 0$.

However, the pixel basis  allows to consider the stationary case where $s$  is a constant and where $T$ tends to $0$.
In fact, from the proof of  Theorem~\ref{thm:conv} we see that
$\chi$ satisfies the partition of unity property if and only if $\sinc(\pi k/ s) = 0$
for every $k \neq 0$. This in turn is the case if and only if
$s = 2^{-m}$ for some $m \in \N$.  The case
$s=1$ seems the most natural one, since it uses non-overlapping basis functions filling the whole space $\R^\dd$.
The non-overlapping case is in fact used in existing
tomographic image reconstruction algorithms; see  \cite{gordon1970algebraic,herman2015basis,kak2001principles}.
Further, note that the number of basis elements  $\chi_{T,s}^k$
for which  its center $m_k \coloneqq Tsk$ is contained  in the
unit cube $[-1,1]^\dd$ is given  by  $(2/(Ts) + 1)^\dd$
and that $T$ is inversely proportional to the essential
bandwidth of the basis function. Therefore,
the choice $s=1$ yields to a minimal number of pixel
basis  functions representing a function with given support
and essential bandwidth.

\subsection{Example:  Generalized Kaiser-Bessel functions}
\label{sec:ex:KB}

As often argued in the literature on tomographic image reconstruction,  the  lack of continuity  and rotation
invariance are severe drawbacks of the  pixel basis functions
for image reconstruction. Therefore in \cite{lewitt1990multidimensional} the generalized Kaiser-Bessel (KB)  functions have been  introduced and proposed for image reconstruction.

The  generalized KB functions in $\R^\dd$ form a family of functions that depend on three parameters $m \in \N$,
$\gamma \geq 0$ and $a >0$,
where $m\in \N$ is referred  to as the order, $\gamma \geq 0$ the taper parameter and $a>0$ is the support parameter. More precisely, the KB function $\ph(\edot; m,\gamma,a ) \colon \R^\dd \to \R $  of order $m$ is defined by
\begin{equation}\label{eq:kb}
\ph(x; m,\gamma,a ) \coloneqq
\begin{cases}
\bkl{\sqrt{1- \norm{x}^2/a^2}}^m \frac{I_m \bkl{\gamma \sqrt{1- \norm{x}^2/a^2}}}{I_m(\gamma)}
& \text{if $\norm{x} \leq a$} \\
0 & \text{otherwise} \,,
\end{cases}
\end{equation}
where $I_m$ is the modified first kind Bessel function.
The window taper  $\gamma$ describes how spiky the basis function is and $a$ is the support radius. The order allows to control the smoothness  and the taper parameter allows to further tune the shape of the basis  function.

The Fourier transform $\hat\ph(\edot; m,\gamma,a )$ of the KB function $\ph(\edot; m,\gamma,a )$ can be computed to  (see \cite{lewitt1990multidimensional})
\begin{equation}\label{eq:kb-hat}
\hat\ph(\xi; m,\gamma,a ) \coloneqq
\begin{cases}
\frac{a^\dd \gamma^m}{I_m(\gamma)} \,
 \frac{I_{d/2+m}  \bkl{\sqrt{\ga^2 - a^2\norm{\xi}^2}}}{\bkl{\sqrt{\ga^2 - a^2\norm{\xi}^2}}^{d/2+m}}
& \text{if $a\norm{\xi} \leq \ga$} \\
\frac{a^\dd \gamma^m}{I_m(\gamma)} \,
 \frac{J_{d/2+m}  \bkl{\sqrt{a^2\norm{\xi}^2-\ga^2 }}}{\bkl{\sqrt{ a^2\norm{\xi}^2-\ga^2}}^{d/2+m}}
& \text{otherwise}\,.
\end{cases}
\end{equation}
Here $J_m$ denotes the first kind Bessel function of order $m$. The known asymptotic   decay
$J_{d/2+m}(r) = \mathcal O(r^{-1/2})$ implies that the asymptotic behavior of the  generalized KB function is
 $\hat\ph(\xi; m,\gamma,a ) = \mathcal O\kl{ \|\xi\|^{-(d/2+m+1/2)}} $.
From Theorem~\ref{thm:conv} we therefore conclude  that  for any choice of $m$, $a$ and $\gamma$, the spaces
\begin{equation*}
\V_{T,s, \ph(\edot; m,\gamma,a )}  =
\overline{\spa\set{\ph_{T,s}^k (\edot; m,\gamma,a )\mid k \in \Z^\dd} }
\end{equation*}
yield  vanishing  approximation error when $s \to 0$ and $T$ keeps bounded. Note that the parameter $a$ plays exactly the same roles as the parameter $T$. Therefore without loss of generality one could omit $a$ in the definition of the KB functions. However we include it since it  is standard  to consider the  KB functions as a  family of three parameters.

Note that the KB function (as any other radially symmetric basis function with compact support) does not satisfy the partition of unity condition. Therefore, Theorem \ref{thm:aerror} implies (for sufficiently regular functions)
that the asymptotic  approximation error saturates; that is, we have
\begin{equation}\label{eq:sat}
\lim_{T \to 0}\norm{\Po_{T,s}f-f}_{L^2}^2
= A_{\ph,s} \norm{f}_{L^2}^2
 \quad \text{ with } \quad
A_{\ph,s}   \coloneqq
\frac{\sum_{k \neq 0}|\hat{\varphi}(2k\pi/s)|^2}{\sum_{k\in\Z^\dd}|\hat{\varphi}(2k\pi / s )|^2}   \,.
\end{equation}
Keeping $m=2$, $a=2$ and $s=1$  fixed, in \cite{nilchian2015optimized} it has been proposed to select the taper parameter $\gamma$ in such a way that the asymptotic approximation error given by $A_{\ph,s}$ is minimized.
Although such a procedure does  not overcome
the saturation   phenomenon, the saturation effect
(for given order and given redundancy factor)
is minimized. Oppose to that, our theory shows that taking $s$
variable and non-constant overcomes the  saturation phenomenon.

In case the partition of unity property does not hold, another natural  strategy to address the  saturation phenomenon  is first  considering $T \to 0$ while keeping $s$ fixed.
In a second step one studies  the saturation error $A_{\ph,s}$
defined in \eqref{eq:sat} as  $s \to 0$.  Similar to
Theorem~\ref{thm:conv}   one can show that the saturation  error vanishes in
limit. More precisely, the following theorem holds.

\begin{theorem}[Saturation error in the  limit]
Let $\ph \in L^2(\R^\dd) \cap L^1(\R^\dd)$ be such that $(\ph_{T,s}^k)_{k\in\Z^\dd}$ is a Riesz-basis of $\V_{T,s,\ph}$ with $\hat{\ph}(\xi)=\mathcal{O}(\|\xi\|^{-p})$ as $\|\xi\|\rightarrow \infty$ for some $p>d/2$, and let $f\in W_2^r(\R^\dd)$ with
$r > d/2$. Then \eqref{eq:sat}  holds and  $\lim_{s \to 0} A_{\ph,s} = 0$.
\end{theorem}

\begin{proof}
We have
\begin{equation*}
 A_{\ph,s} = \frac{\sum_{k\neq 0}|\hat{\ph}(2k\pi/s)|^2}{\sum_{k\in\Z^\dd}|\hat{\ph}(2k\pi/s)|^2}=1-\frac{|\hat{\ph}(0)|^2}{|\hat{\ph}(0)|^2 + \sum_{k\neq 0}|\hat{\ph}(2k\pi/s)|^2} \,.
\end{equation*}
Therefore the claim  follows  after showing
$\sum_{k\neq 0}|\hat{\ph}(2k\pi/s)|^2 \to 0$ as $s \to 0$.
Since $\hat{\ph}(\xi)=\mathcal{O}(\|\xi\|^{-p})$,  there exist
$C,s_0>0$  with $|\hat{\ph}(2k\pi/s)|\leq C \|2k\pi/s\|^{-p}$
for  $k\neq 0$ and $s \leq s_0$. This implies  $\sum_{k\neq 0}|\hat{\ph}(2k\pi/s)|^2 \leq C(\frac{s}{2\pi})^{2p}\sum_{k\neq 0}\| k\|^{-2p}$. Because $p>d/2$, the  sum $\sum_{k\neq 0}\| k\|^{-2p}$ is absolutely convergent. Hence we have  $\sum_{k\neq 0}|\hat{\ph}(2k\pi/s)|^2 \to 0$ as $s \to 0$
which concludes the proof.
\end{proof}

\section{The Galerkin approach for PAT using shift invariant spaces}
\label{sec:galerkinshift}

In this section we give details how to efficiently
implement the  least squares Galerkin method using
subspaces of a shift invariant space.
This is in contrast to the use of a general reconstruction space,
where both the computation of the system matrix and the solution of the
 Galerkin equation can be  slow.   For shift invariant spaces  the
system matrix takes a very special form which
allows an efficient implementation.

Let $\ph \in L^2(\R^\dd)$ be such that
the elements $\ph_{T,s}^k$
form a Riesz basis of
$\V_{T,s,\ph}$; see Section~\ref{sec:spaces}.
Moreover, let $(T_N)_{N \in \N}$ and $(s_N)_{N \in \N}$ be two
sequences of positive numbers describing the support  and the  redundancy of the basis functions, respectively.
We consider the
reconstruction spaces
\begin{equation} \label{eq:XN}
\X_N  \coloneqq
\left\{\sum_{k\in\La_N}  c_k \ph_N^k \mid
  k  \in \La_N \right\}
\subseteq  \V_{T_N,s_N} \,,
\end{equation}
where $\ph_N^k \coloneqq \ph_{T_N,s_N}^k$ are the basis functions
(with $\ph_{T_N,s_N}^k$ as in  \eqref{eq:phik}), and $\La_N \coloneqq \{k\in\Z^\dd \mid   m_k \coloneqq T_N s_N k  \in B_R(0) \}$
denotes the set of all $k \in  \Z^\dd$ such that  the mid-point $m_k$ of the $k$-th basis function  is  contained in $B_R(0)$. Then $\dim \X_N = |\Lambda_N|$ is the number of basis elements used for image representation.   In the case that the support of the function to be reconstructed  intersects (or is at least close) to $\partial B_R(0)$, it may be better to use all basis functions $\ph_{T_N,s_N}^k$
whose support intersects  $\overline{B_R(0)}$. The following consideration also hold  for such an alternative choice. Further note that the approximation results  for the shift invariant spaces $\V_{T,s}$ do not  immediately  yield approximation results the  finite dimensional spaces $ \X_N$. Such investigates are an interesting  aspect  of future studies.

When applied with the reconstruction  space $\X_N $, our
Galerkin approach to PAT analyzed in Section~\ref{sec:galerkinPAT}
takes the form (see Theorem~\ref{thm:leastsquares})
 \begin{equation} \label{eq:galerkin2}
 f_N  = \sum_{k\in \La_N} c_{N,k} \ph_N^k \,,
 \end{equation}
 where
\begin{itemize}
\item $\Ao_N \coloneqq ( \tfrac{R}{2}\ip{\ph_N^k}{\ph_N^\ell}_{L^2})_{k,\ell\in \La_N}$ is the system matrix ;
\item $d_N \coloneqq (\ip{\Wo \ph_N^k}{g}_t)_{k\in \La_N}$ is the right hand side;
\item $c_N \coloneqq (c_{N,k})_k$ solves the Galerkin
equation $\Ao_N c_N = d_N$.
\end{itemize}
As discussed in  the following subsection,  for the shift invariant case the  system matrix $ \Ao_N$ takes a very special form which significantly  simplifies the computations.  Further,  the right hand of the Galerkin equation can be computed efficiently as described in Subsection \ref{sec:rhs} below.

\subsection{Evaluation of the system matrix}
\label{sec:sys}

For any $N \in \N$ and
any $k, \ell \in \Lambda_N$, the
entries  of the system matrix $\Ao_N$ satisfy
\begin{align*}
	\ip{\ph_N^k}{\ph_N^\ell}
	&=
	\frac{1}{T_N^\dd}
	\int_{\R^\dd}
	\ph\bkl{\frac{x}{T_N} - s_N k}
	\ph\bkl{\frac{x}{T_N} - s_N \ell}
	\rmd x \\
	&=
	\int_{\R^\dd}
	\ph\bkl{y - s_N k}
	\ph\bkl{y - s_N \ell}
	\rmd y \\
	&=
	\int_{\R^\dd}
	\ph\bkl{x}
	\ph\bkl{x - s_N (\ell-k)}
	\rmd x \\
	&=  \ip{\ph_{1,s_N}^0}{\ph_{1,s_N}^{\ell-k}}\,.
\end{align*}
Hence instead of computing and storing the  whole system
matrix required by standard Galerkin methods, in our approach
only the values   $\ip{\ph_{1,s_N}^0}{\ph_{1,s_N}^{n}}$
where $n = \ell-k$ with $\ell, k \in \Lambda_N$
have to be computed and stored.
 The total number of such inner products is bounded by
 $2^\dd |\Lambda_N|$.
In the case  where $\ph$ has small support this
number is actually  much smaller since
$\ip{\ph_{1,s_N}^k}{\ph_{1,s_N}^\ell}$ vanishes if the
supports of $\ph_{1,s_N}^k$ and $\ph_{1,s_N}^\ell$
do not overlap.

In this paper we mainly consider the (non-overlapping) pixel basis (see Subsection~\ref{sec:ex:pixel})
and the KB functions in two spatial dimensions (see Subsection~\ref{sec:ex:KB}).  The pixel basis is an orthonormal
system and therefore the system matrix is the identity. The KB functions are radially symmetric. In such a situation we compute
the  entries $\ip{\ph_{1,s_N}^0}{\ph_{1,s_N}^{\ell-k}}$ of the system matrix $\Ao_N$ approximately as follows.
We numerically computed the inner products $\langle \varphi_{1,s_N}^0,\varphi_{1,s_N}^k\rangle_{L^2}$ for all $k\in\Z^2$
with $\|k\|_2\leq 2a$ using the rectangle rule.
For this we discretized the square $[-a,a]^2$ by an equidistant Cartesian grid with
$M \times M$ grid points $(x_i,y_j)$ and computed
\begin{equation}\label{eq:sm-discrete}
\langle \varphi_{1,s_N}^0,\varphi_{1,s_N}^k\rangle_{L^2}
\simeq \frac{(2a)^2}{(M-1)^2}\sum_{i=1}^{M}\sum_{j=1}^{M}\varphi_{1,s_N}^{0}(x_i,y_j)\varphi_{1,s_N}^{k}(x_i,y_j) \,.
\end{equation}
The resulting system matrix is a tensor product of Toeplitz matrices.

\subsection{Evaluation of the right hand side}
\label{sec:rhs}

In the practical  application instead of the continuously sampled data $g = \Wo f$ only
discrete data  $g(z_i, t_j)_{i,j}$ are known, where
$t_j = j \, T/\Nt$ are $\Nt$ equidistant  time points  in the interval  $[0,T]$  and  $z_i$ are $\Ndet$ points  on the
measurement surface $\partial B_R(0)$.
In our numerical  implementation we approximate the right hand side
in the Galerkin equation as follows:
\begin{equation}\label{eq:ipg1}
	\ip{\Wo\ph_N^k}{g}_t
	\simeq
     \frac{T}{\Nt-1} \sum_{i=1}^{\Ndet}
     \sum_{j=1}^{Nt}
     w_i  (\Wo\ph_N^k)(z_i, t_j) g(z_i, t_j) t_j \,.
\end{equation}
Here $w_i$ are appropriate weights accounting for the density of the sampling points.
The right hand side in~\eqref{eq:ipg1} may be interpreted as the exact inner
product $\ip{\Wo\ph_N^k}{g^\delta}_t$ for some approximate data
$g^\delta \simeq g$, which  allows application of  our convergence
and stability result derived in Theorem~\ref{thm:conv}.

In some situations (for example for the KB functions and other radially symmetric basis functions in three dimensions), the solution of  $\Wo\ph_N^k$ is available analytically (see~\cite{diebold1991photoacoustic,wang2014discrete}).  In our numerical solutions we  use the pixel basis and the KB basis  functions in two spatial dimensions, where  we are not aware of  explicit representations for the corresponding solution of the wave equation.  In this case we numerically
compute $\Wo\varphi$ using the well known
solution formula for the wave  equation \eqref{eq:wave-fwd},
\begin{align}\label{eq:wavesol}
\Wo f(z,t)
=
(\partial_t \Ao_t \Mo f) (z,t)
\coloneqq
  \frac{1}{2\pi}
\frac{\partial}{\partial t}
\int_0^t\int_{\sph^1} \frac{r f(z +  r \omega)}{\sqrt{t^2-r^2}}
\rmd s(\omega) \rmd r \,.
\end{align}
Here
\begin{align*}
	&\forall (z,r) \in \partial B_R(0) \times (0, \infty) \colon
	&&
	\Mo f(z,r)
	\coloneqq\frac{1}{2\pi}\int_{\sph^1}
	f(z +  r \omega) \rmd s(\omega)  \,,\\
	&\forall (z,t) \in \partial B_R(0) \times (0, \infty) \colon
	&&
	\Ao_t g(z,t)
	\coloneqq
	\int_0^t \frac{r g(z,r)}{\sqrt{t^2-r^2}} \, \rmd r \,,
\end{align*}
denote the spherical means transform of a function $f \colon \R^2 \to \R$  with support in  $B_R(0)$,  and the Abel transform of a function
$g \colon \partial B_R(0) \times (0, \infty) \to \R$ in the second variable, respectively.
The solution formula  \eqref{eq:wavesol} is used to numerically compute $\Wo\ph_N^k$ required for evaluating the
right hand side of the Galerkin  equation
as outlined in the following.

\begin{itemize}

\item
For a  symmetric basis function of the form $\varphi(x) = \bar \ph(\norm x)$ the corresponding solution of the wave equation also is radially symmetric. Hence in order to approximate
$\Wo\varphi_N^k$ we  numerically approximate $\Wo \varphi_{1,s_N}^0((r_n,0), t_j)$
for $\Nrad$ equidistant  radii $r_n \in[0,2R]$ and
using a numerical approximation of $\Wo$
by discretizing   the spherical Radon transform as well
as the Abel transform in \eqref{eq:wavesol}.
As a next step, for any basis functions $\varphi_N^k$,
we approximately compute
\begin{equation*}
	\Wo\varphi_N^k(z_i,t_j)
	=
	\Wo\varphi_{1,s_N}^{0}((\|z_i - k\|,0), t_j)
\end{equation*}
at any detector points $z_i \in \partial B_R(0)$
 and discrete time points $t_j$  by replacing the right
hand side with the piecewise linear interpolation in the first argument
using the known values $\Wo\varphi_{1,s_N}^0((r_n,0), t_j)$.

\item In the case of the pixel  basis, the spherical means
$\Mo \chi_N^k$ have been computed analytically and
evaluated  at  the discretization points $(z_i,t_j)$.
Subsequently, the wave data $\Wo\varphi_N^k(z_i,t_j)$
are computed  by numerically evaluating the Abel transform in~\eqref{eq:wavesol}.
 \end{itemize}

\begin{remark}[Implementation\label{rem:FE} for finite element bases] 
Above we have demonstrated  how the Galerkin approach  can be 
implemented efficiently if $\X_N$ is generated by a radially symmetric function.
In fact,  an  efficient implementation of the Galerkin method for PAT can be obtained  for an arbitrary generating element.
In this case the entries of the system  matrix are  computed  similar to~\eqref{eq:sm-discrete}.
Due to the lack of symmetry of the basis functions,  the solution formula \eqref{eq:wavesol}  cannot be used
to accelerate the computation of $\Wo \ph_N^k$; still this would require  the separate  evaluation \eqref{eq:wavesol} for each basis function.
In the non-radially symmetric case, however, one 
can exploit again that $\ph_N^k(x) = \ph((x - skT)/T)$ are translates of a 
single function  $\ph_{T,1}^0$.
For that  purpose, one first numerically computes the  solution  $p(x,t)$  of the wave equation \eqref{eq:wave-fwd}
with initial data $\ph_{T,1}^0$. In a second step one uses interpolation to approximately find $(\Wo \ph_N^k) (z_i,t_j) = p(z_i -  skT, t_j)$.

Such an approach can, for example, be used for a bilinear finite elements basis that consists of scaled and translated versions of the basis function $\ph \colon \R^2\rightarrow\R$  defined by
\begin{equation}
\ph(x)  \coloneqq
\begin{cases}
(1-\abs{x_1})(1-\abs{x_2}) &(x_1,x_2)\in [-1,1]^2  \\
0  &\text{else} \,. \end{cases}
\end{equation}
The corresponding finite element basis using the shift parameter $s=1/2$
can easily be seen to satisfy the partition of unity property. Note that in this case the
entries of the system  matrix even can be computed analytically. Exploring the use of finite elements further is  beyond the scope of this paper.  However, we think that a precise comparison of different basis elements
(in combination with our Galerkin approach as well as in combination with related approaches) is an interesting and important line of future research.
\end{remark}

\section{Numerical studies}
\label{sec:num}

In this section we present results of our numerical studies  for
our Galerkin least squares approach,  where the approximation  space
$\X_N$ is taken as the subspace  of  a shift  invariant space. We further compare our results with related approaches in the
literature.   We restrict ourselves to the case of two spatial dimensions and take $R=1$ for the radius of the  measurement circle.

For all presented  numerical results, the function $f$ is taken a superposition of  indicator functions as shown in top left
 image in Figure~\ref{fig:rec1}.
The  corresponding discrete  data
\begin{equation} \label{eq:ddatad}
g(z_i, t_j)  \simeq (\Wo f)(z_i, t_j)
\quad
\text{ for   $i = 1, \dots \Ndet$
and $j = 1, \dots , \Nt$} \,,
\end{equation}
where  $z_i = (\cos (i 2\pi/\Ndet ), \sin (i 2\pi/\Ndet ))$ denote the equidistant detector locations and $t_j = j T/\Nt$ the discrete time points, have been computed numerically by implementing \eqref{eq:wavesol}.  For that purpose
we discretized the spherical Radon transform as well  as the Abel transform in \eqref{eq:wavesol}.
We take $T=3$ as the final measurement time,
$\Ndet = 100$, $\Nt  = 376$ for the discretization of the
data and $\Nx = 300$ for discretizing the function.
Note that the data are computed in a way that is completely
different from the Galerkin system which avoids any
inverse  crime.

\subsection{Reconstruction results  using Kaiser-Bessel Functions}

We first investigate the case where $\ph =  \ph(\edot; m,a,\gamma)$
is a  KB function. The parameters $m$ and $\gamma$
determine the shape and smoothness of the KB function, whereas $a$
determines its support. It is therefore reasonable to fix $m$ and $\gamma$.
Here we choose the fixed parameters  $m=1$ and $\gamma=2$. Further,
$a$ determines the support of the KB function, which  is also controlled by the parameter $T$. Therefore also this parameter can be fixed; without loss of
generality we take $a=2$. This effects that for $s=1$ the functions
$\ph_N^k$ show sufficient overlap.
Since the total number of basis functions which are centered in the square $[-1,1]^2$ is equal to $N^2$ with $N = 2 /(sT)+1$ it is reasonable to consider combinations of the parameters $s$ and  $T$  where the product $sT$ remains constant.

\begin{figure}[thb!]
\centering
\includegraphics[width=0.45\textwidth, height=0.36\textwidth]{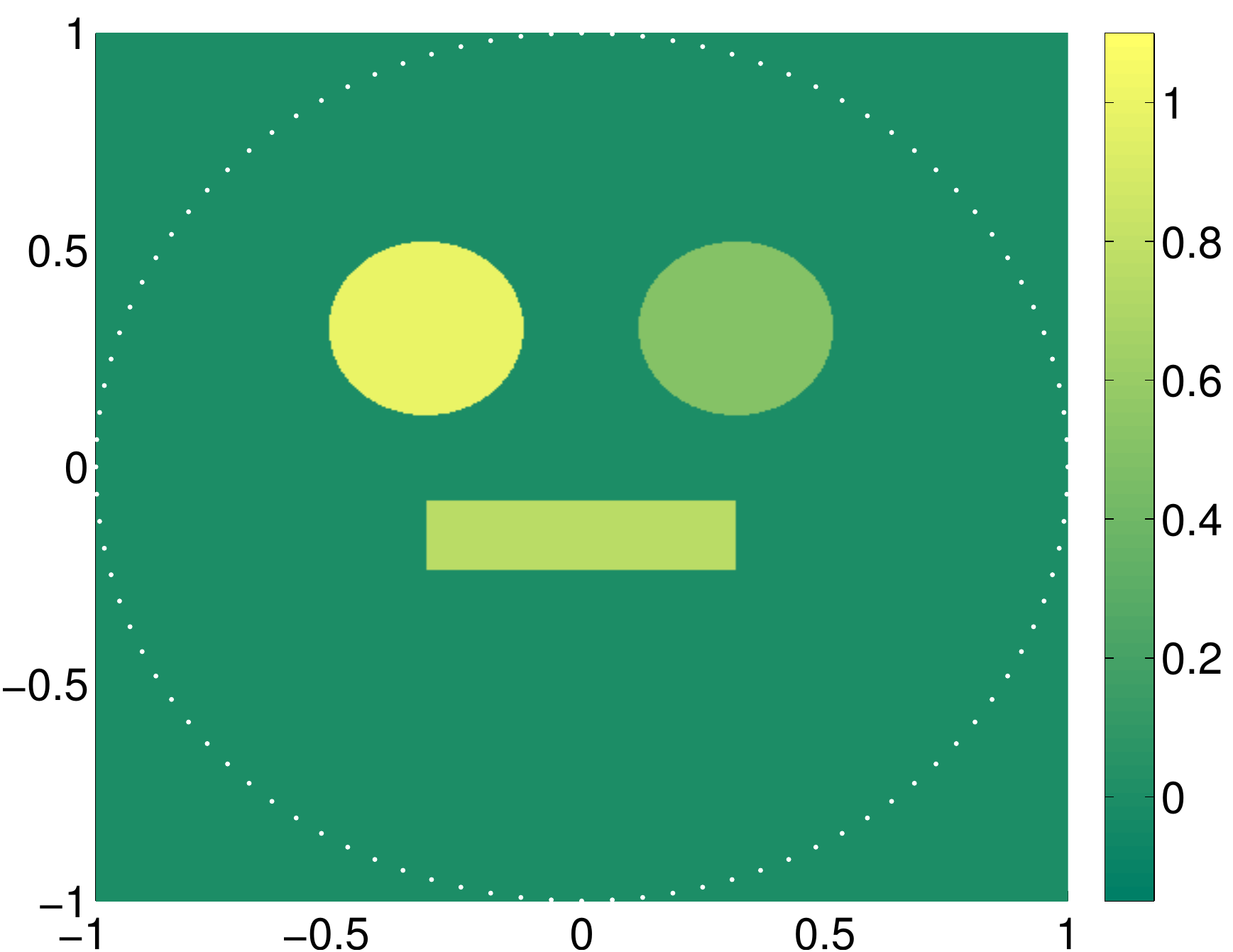}\quad
\includegraphics[width=0.45\textwidth, height=0.36\textwidth]{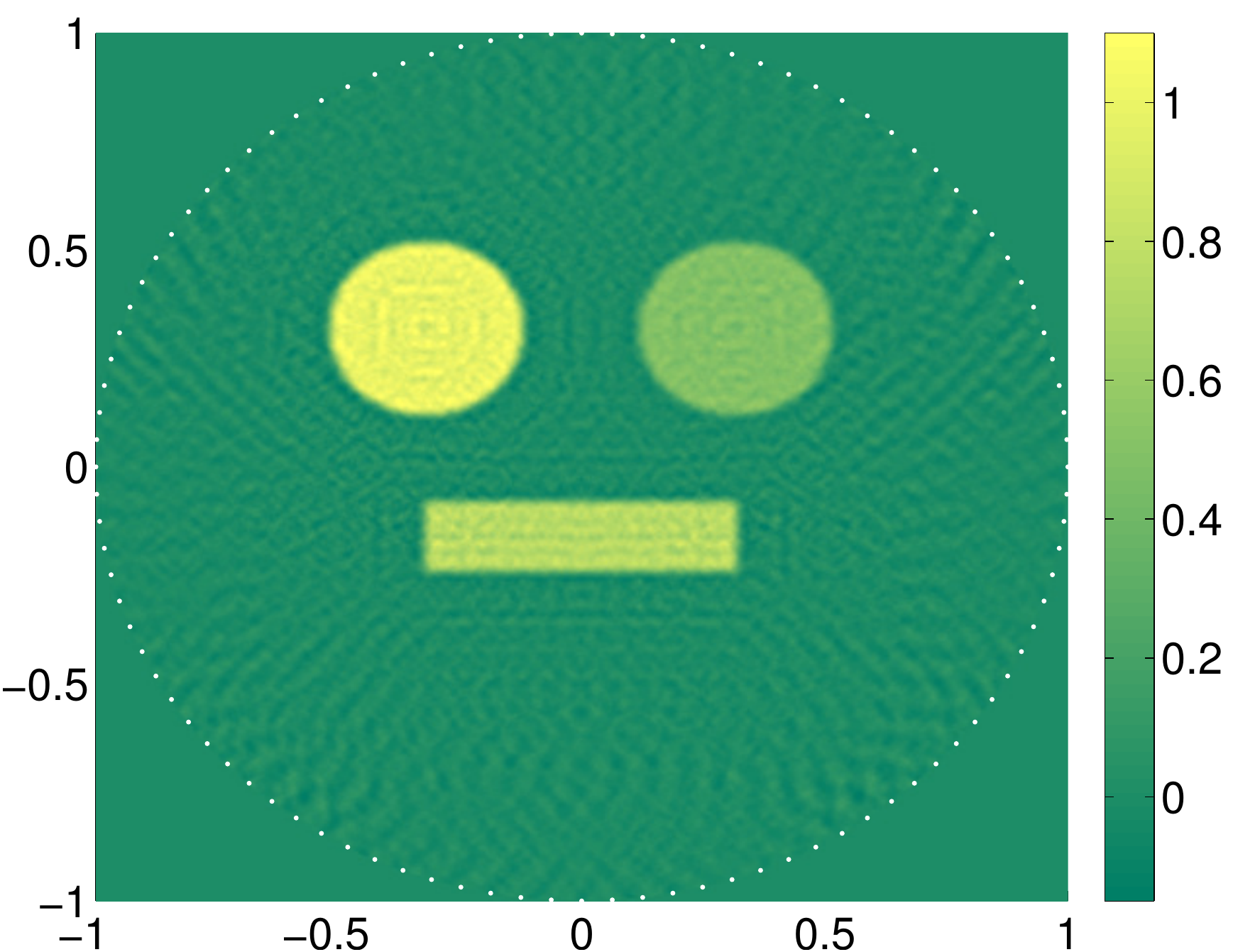}
\\[0.5em]
\includegraphics[width=0.45\textwidth, height=0.36\textwidth]{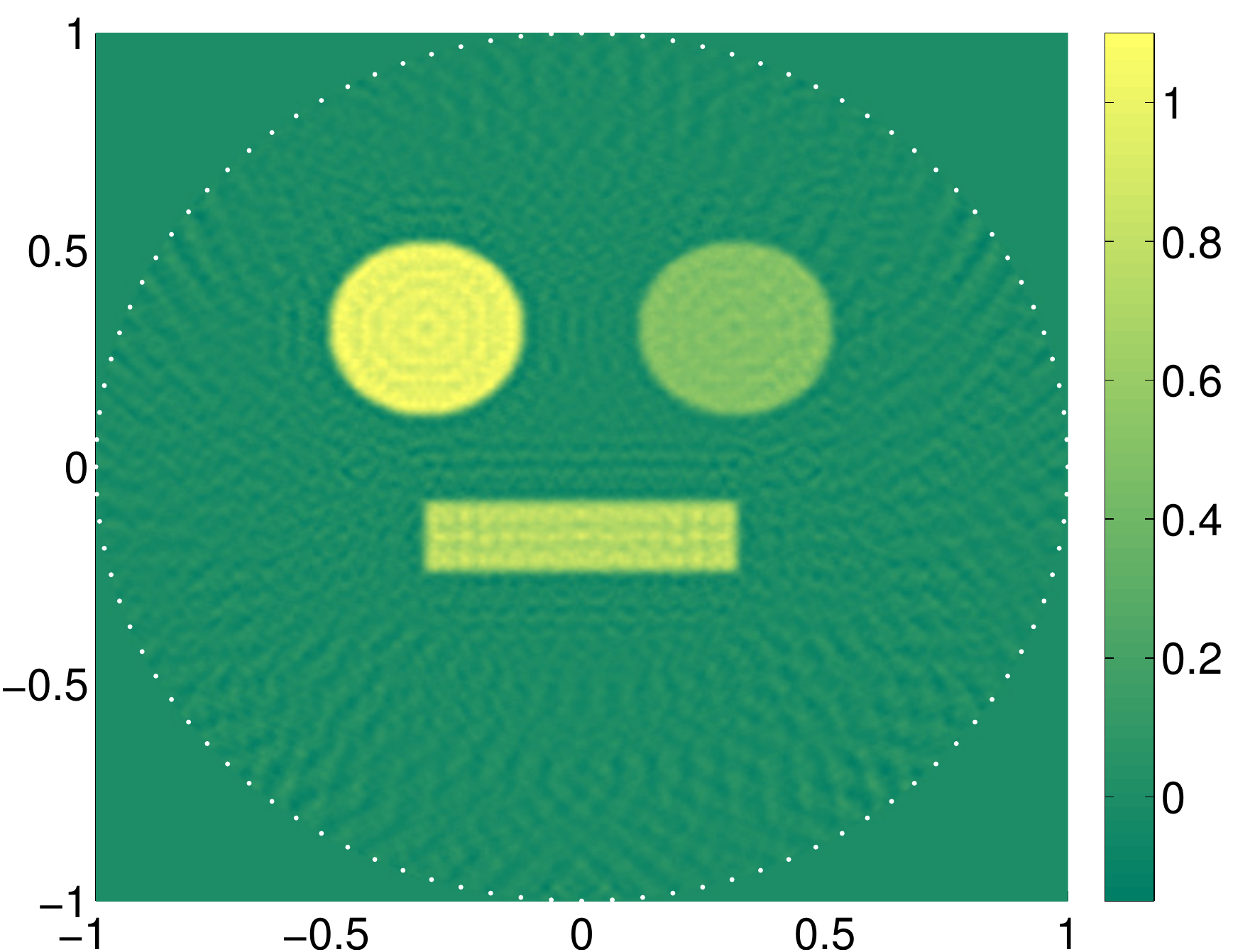}\quad
\includegraphics[width=0.45\textwidth, height=0.36\textwidth]{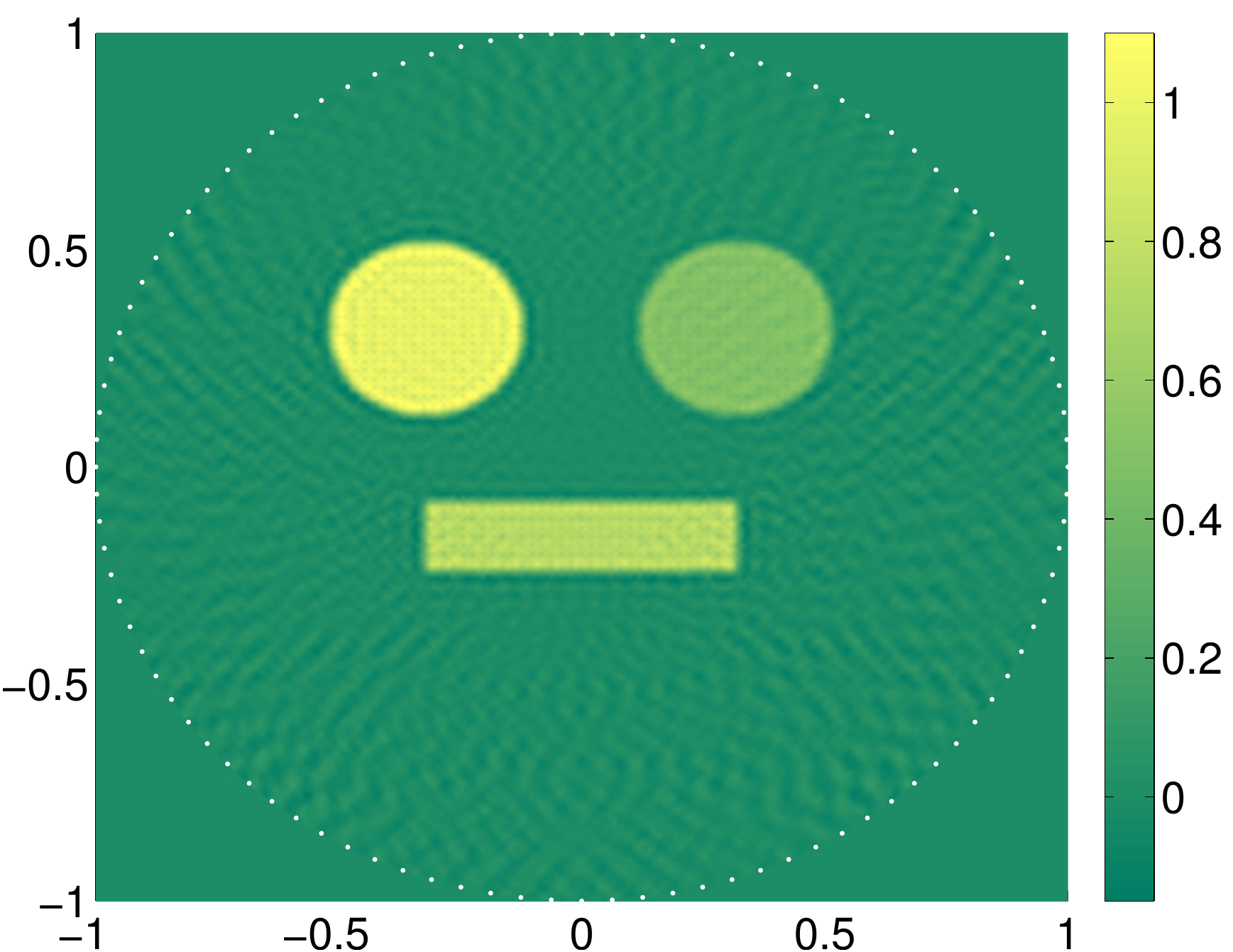}
\caption{\textsc{Reconstruction results\label{fig:rec1}  using the proposed KB Galerkin approach.} Top left:   Phantom $f$.
Top right: Reconstruction with $s=0.8081$, $T=0.025$.
Bottom left: Reconstruction with $s=1.0101$, $T=0.02$.
Bottom right: Reconstruction with $s=1.3636$ and $T=0.0148$.}
\end{figure}

The proposed KB Galerkin approach for the inverse PAT problem consists in  solving the Galerkin equation~\eqref{eq:galerkin2}. Therefore the  system matrix and the right hand side are computed in \MATLAB
as described in Section~\ref{sec:galerkinshift} and the direct solver {\tt mldivide} is used for numerically computing the solution of  \eqref{eq:galerkin2}.
Figure~\ref{fig:rec1} shows reconstructions using the  KB Galerkin reconstruction
for $N=100$  and step size parameters $s=0.8081$,
$s = 1.0101$ and $s=1.3636$, respectively. One  notices
that actually all considered step size parameters  yield
quite good results.


\begin{table}[thb!]
\centering
\begin{tabular}{c c c | c c c }
\toprule
 \multicolumn{3}{c}{$N=50$} & \multicolumn{3}{c}{$N=100$} \\
\midrule
$s$ & $T$ & $e_N(s,f)$ &  $s$ & $T$ & $e_N(s,f)$  \\
\midrule
1.4286    &0.0286& 0.0412  &1.4141 &0.0143& 0.0336 \\
\midrule
1.3776  &0.0296& 0.0379  &1.3636 &0.0148& 0.0299 \\
\midrule
1.3265 &0.0308& 0.0351   &1.3131 &0.0154& 0.0298  \\
\midrule
1.2755  &0.0320& 0.0352   &1.2626 &0.0160&  0.0303  \\
\midrule
1.2245   &0.0333&  0.0369   &1.2121 & 0.0167&  0.0307  \\
\midrule
1.1735 &0.0348&  0.0401    &1.1616  &0.0174&  0.0320    \\
\midrule
1.1224  &0.0364&  0.0439   &1.1111  &0.0182& 0.0345  \\
\midrule
1.0714  &0.0381& 0.0427    &1.0606 & 0.0190&  0.0367   \\
\midrule
1.0204  &0.0400&  0.0428     &1.0101& 0.0200&  0.0384   \\
\midrule
0.9694 &0.0421& 0.0403    &0.9596  & 0.0211& 0.0335   \\
\midrule
0.9184 &0.0444& 0.0391    &0.9091&0.0222&  0.0392  \\
\midrule
0.8673  &0.0471& 0.0412  &0.8586 &0.0235&  0.0366    \\
\midrule
0.8163   &0.0500& 0.0432  &0.8081&0.0250& 0.0322   \\
\midrule
0.7653    &0.0533& 0.0464    &0.7576 & 0.0267&  0.0365    \\
\bottomrule
\end{tabular}
\caption{\textsc{\label{tab:error} Relative $L^2$-reconstruction errors} For different choices of $s$ the reconstruction error  $e_N(s,f)$ with is evaluated for $N=50$ and $N=100$. Recall that $s$ is the step size and $T = 2/(s(N-1))$ determines the  size of the KB basis functions.}
\end{table}

\subsection{Parameter selection for the KB  functions}

Choosing  optimal parameters seems a difficult issue.
In the following we numerically investigate the optimal choice of the  parameters $s$ and $T$ for a fixed number of basis functions $N^2$ with $N=50$ and $N=100$, respectively.  For that purpose we compute the $L^2$-reconstruction  error
 \begin{equation} \label{eq:err-nsf}
	e_N(s,f)  \coloneqq
	\frac{\sum_{i=1}^N\sum_{j=1}^N
	|f_N(x_i,y_j)-f(x_i,y_j)|^2}{\sum_{i=1}^N\sum_{j=1}^N |f(x_i,y_j)|^2}.
\end{equation}
for different choices of  $s$ and $T$ satisfying  the side condition
$sT = 2/(N-1)$.
Here $f_N \coloneqq\sum_{k\in\Lambda_N} c_{N,k}\varphi_N^k$  is the Galerkin reconstruction
given by \eqref{eq:galerkin2} and the evaluation points $(x_i,y_j)$  for evaluating the error in
\eqref{eq:err-nsf}
are taken as  the elements on $\{s_N T_N k \  \mid \ k\in\Z^2\}\cap[-1,1]^2$.
In Table~\ref{tab:error} we show these relative $L^2$ reconstruction errors.
From  Table~\ref{tab:error} one finds that for the considered function optimal  choices for the step size parameter  are
$s=1.3265$ for $N=50$ and  $s = 1.3131$ for $N=100$.

\begin{figure}[thb!]
\centering
\includegraphics[width=0.45\textwidth]{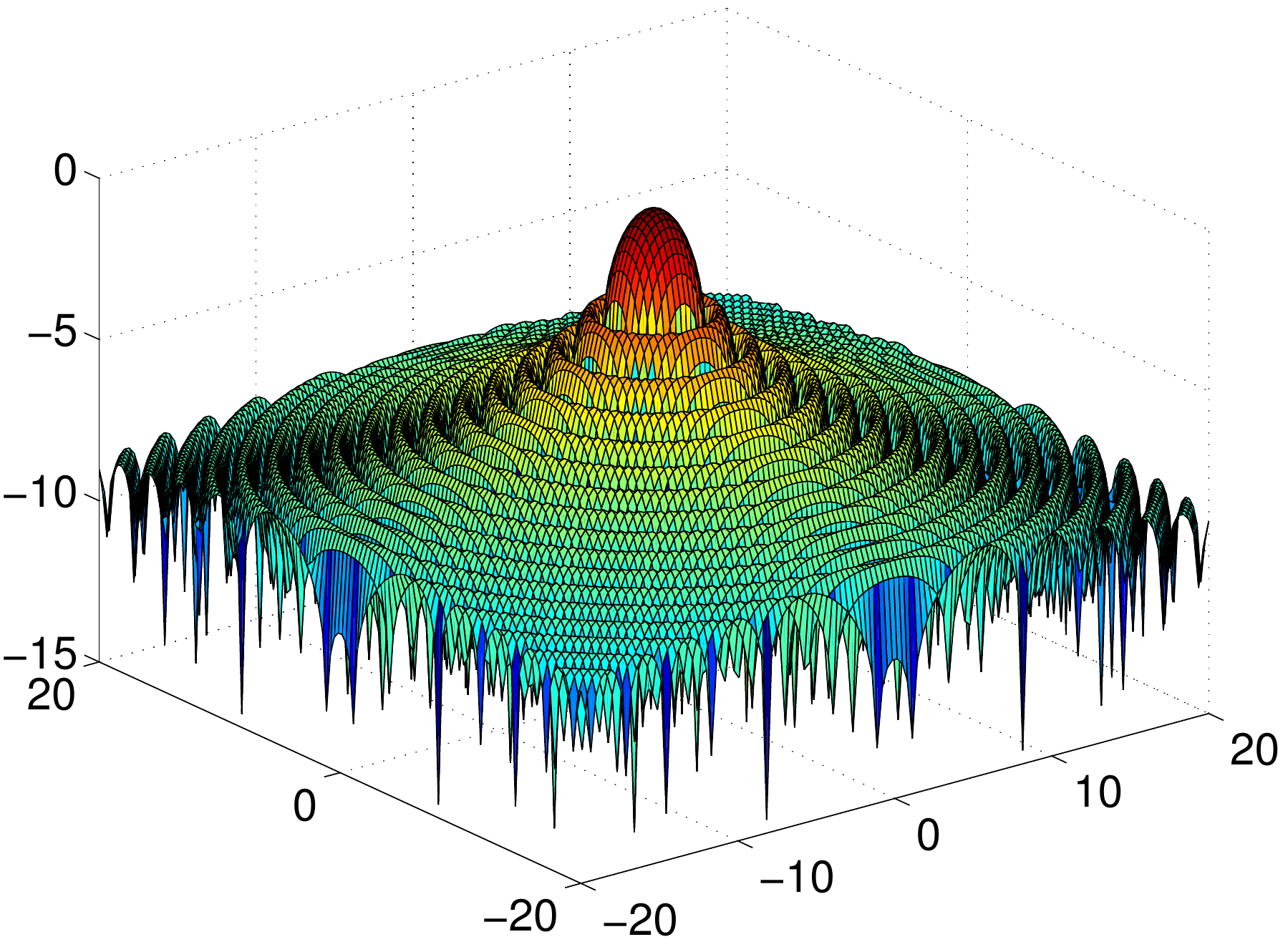} \quad
\includegraphics[width=0.45\textwidth]{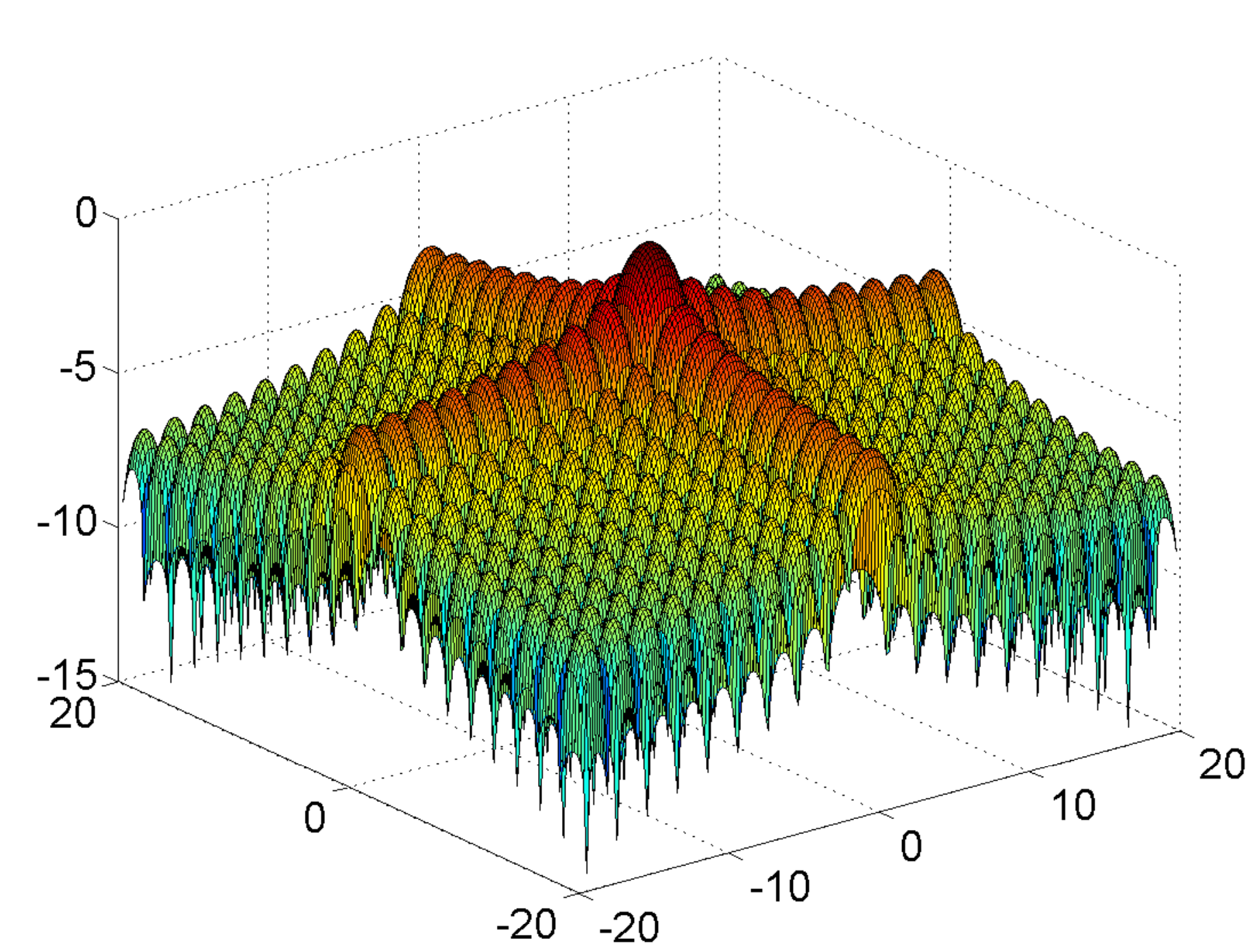}
\caption{\textsc{Logarithmic\label{fig:psi} plots of the Fourier transforms of the basis functions.}
Left: Fourier transform $\log( \abs{\hat{\varphi}(\edot  ; 1,2,2)})$ of the
KB basis function. Right: Fourier transform  $\log(\abs{\sinc(\edot)})$ of the pixel basis function
$\chi$.}
\end{figure}

From Table~\ref{tab:error} one  notices an irregular behavior of the
reconstruction error in dependance on $s$  and $N$.
To better understand this  issue recall  that  for a given basis function
$\varphi = \varphi(\edot ; m,a,\gamma)$  the
best the $L^2$-approximation error
 using functions $\ph_{T,s}^k$ is given by (see Theorem~\ref{thm:aerror})
\begin{equation*}
\norm{\Po_{T,s}f-f}_{L^2}
=
\int_{\bigl[-\frac{\pi}{Ts},\frac{\pi}{Ts}\bigr]^2}\abs{\hat{f}(\xi)}^2
\bkl{ 1-  \frac{\abs{\hat{\varphi}(T\xi)}^2}{\sum_{k\in\Z^\dd}\abs{\hat{\varphi}(T\xi+2k\pi/s)}^2}}   \,,
\end{equation*}
where it is  assumed  that the Fourier transform of $f$ is sufficiently small outside
$[-\pi/(Ts),\pi/(Ts)]^2$. Hence  for fixed  $N$  a ``good'' choice of $s$
 should be made at least in such a way that
\begin{equation*}
S_{\ph}(s,N ,\xi )
\coloneqq
 \frac{\sum_{k \neq 0} \abs{\hat{\varphi} \bigl( \frac{2}{s (N-1)} \xi - \frac{2k\pi}{s} \bigr) }^2}
 {\abs{\hat{\varphi}\bigl(\frac{2}{s (N-1)} \xi  \bigr)}^{2}}
\text{is ``small'' for  } \norm{\xi} \leq \frac{\pi (N-1)}{2} \,.
\end{equation*}
(We have taken $T  = 2/(s (N-1))$ and $\hat f$ is supposed to be unknown.)
Figure~\ref{fig:psi} shows that absolute value of the radially symmetric Fourier
transform of the basis function $\varphi(\edot ; 1,2,2)$ in a logarithmic plot.
This shows a complicated dependence of $S_{\ph}(s, N ,\xi )$ on $s$, $N$  and $\xi$ and
indicates that a simple universally valid answer how to optimally chose parameters seems difficult.
We further note that $S_{\ph}(s,N ,\xi )$ does not contain error due to frequency content
outside  $[-\pi/(Ts),\pi/(Ts)]^2$.  Nevertheless, theoretical error estimates
in combination  with numerical studies can give precise guidelines for selecting good parameter
for the practical applications.
The quality of the reconstruction depends on the parameters of the KB function $m,\gamma, a$ as well as on $s$ and $T$ (note that $T$ has a similar role as $a$).
In the paper \cite{nilchian2015optimized} the authors studied optimizing the parameter
$\gamma$ (in the limit $T \to 0$) while the parameters $s=1$, $a=2$ and $m=2$ have been kept fixed.  For that purpose they choose the parameter $\gamma$ in
$\varphi(\cdot \ ; 2,2,\gamma)$ such that the limiting residual error $S_{\ph}(1, 0 ,\xi ) = \sum_{k \neq 0} |\hat{\varphi}(2k\pi)|^2$ (that is independent of $\xi$) becomes minimal. As we argued  above the drawback of such an approach is that taking $s$ fixed does not yield vanishing asymptotic error as
$N \to \infty$.  Allowing $s$ to depend on $N$ overcomes this issue but makes the parameter selection more complicated.

\begin{figure}[thb!]
\centering
\includegraphics[width=0.45\textwidth, height=0.36\textwidth]{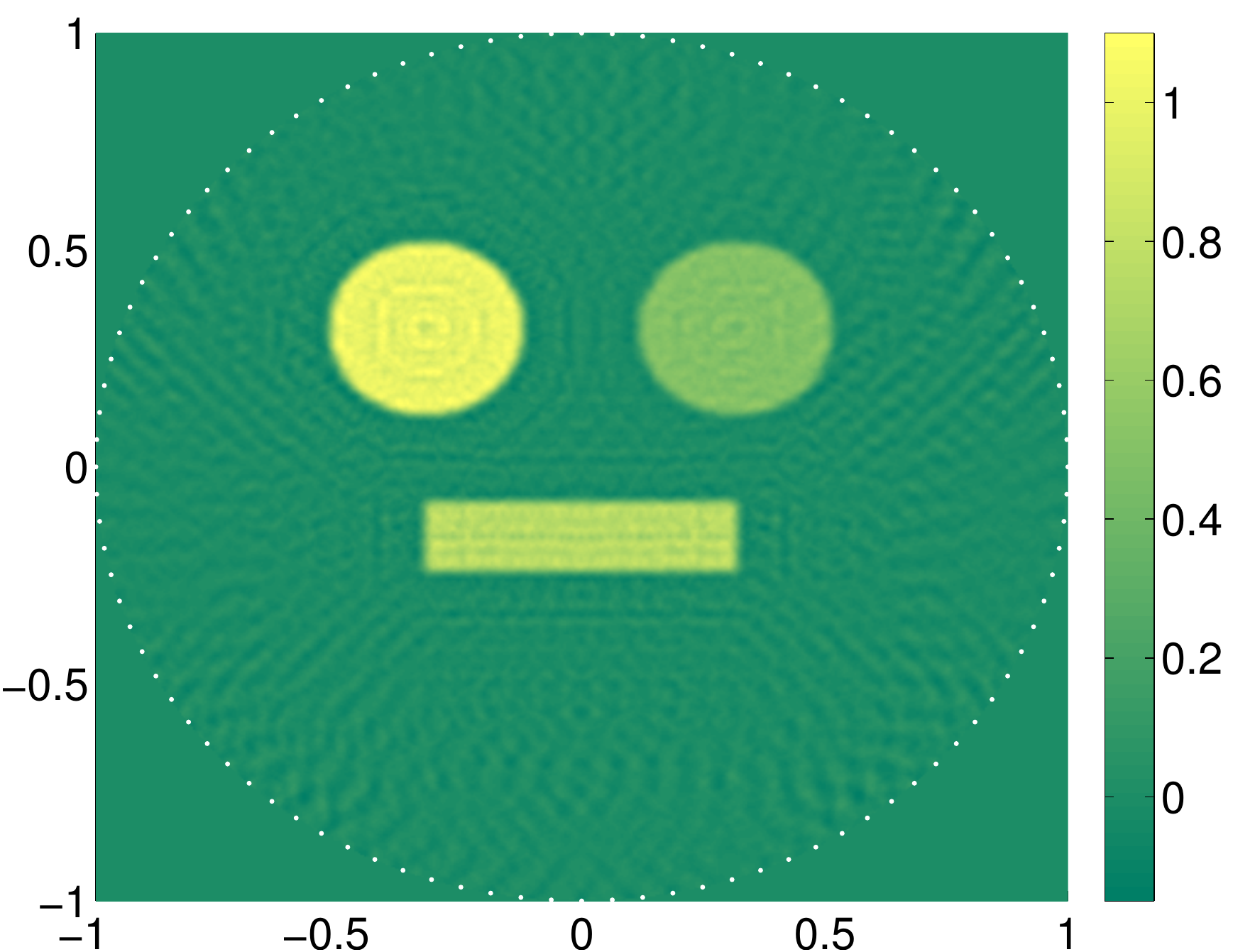}\quad
\includegraphics[width=0.45\textwidth, height=0.36\textwidth]{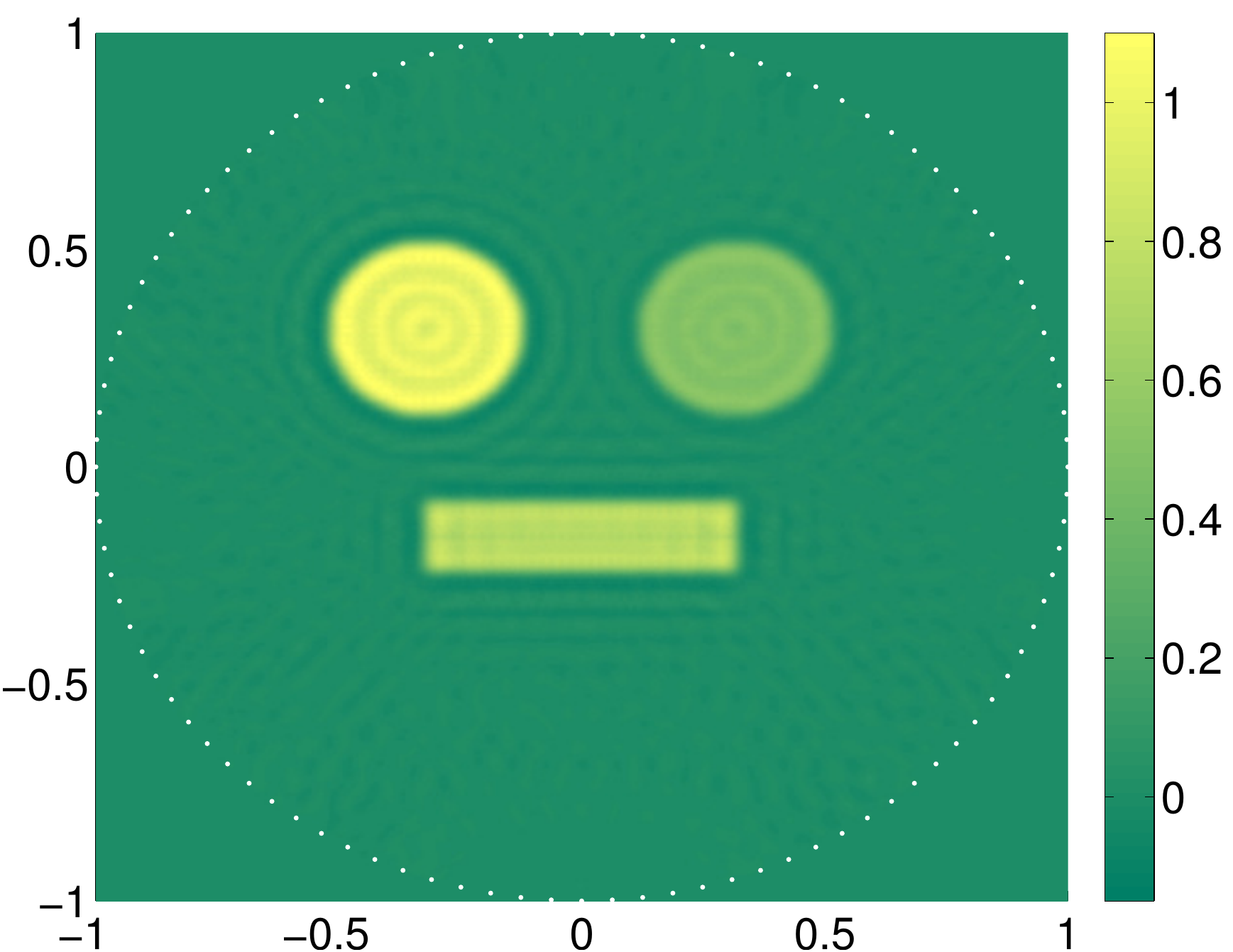}\\[0.5em]
\includegraphics[width=0.45\textwidth, height=0.36\textwidth]{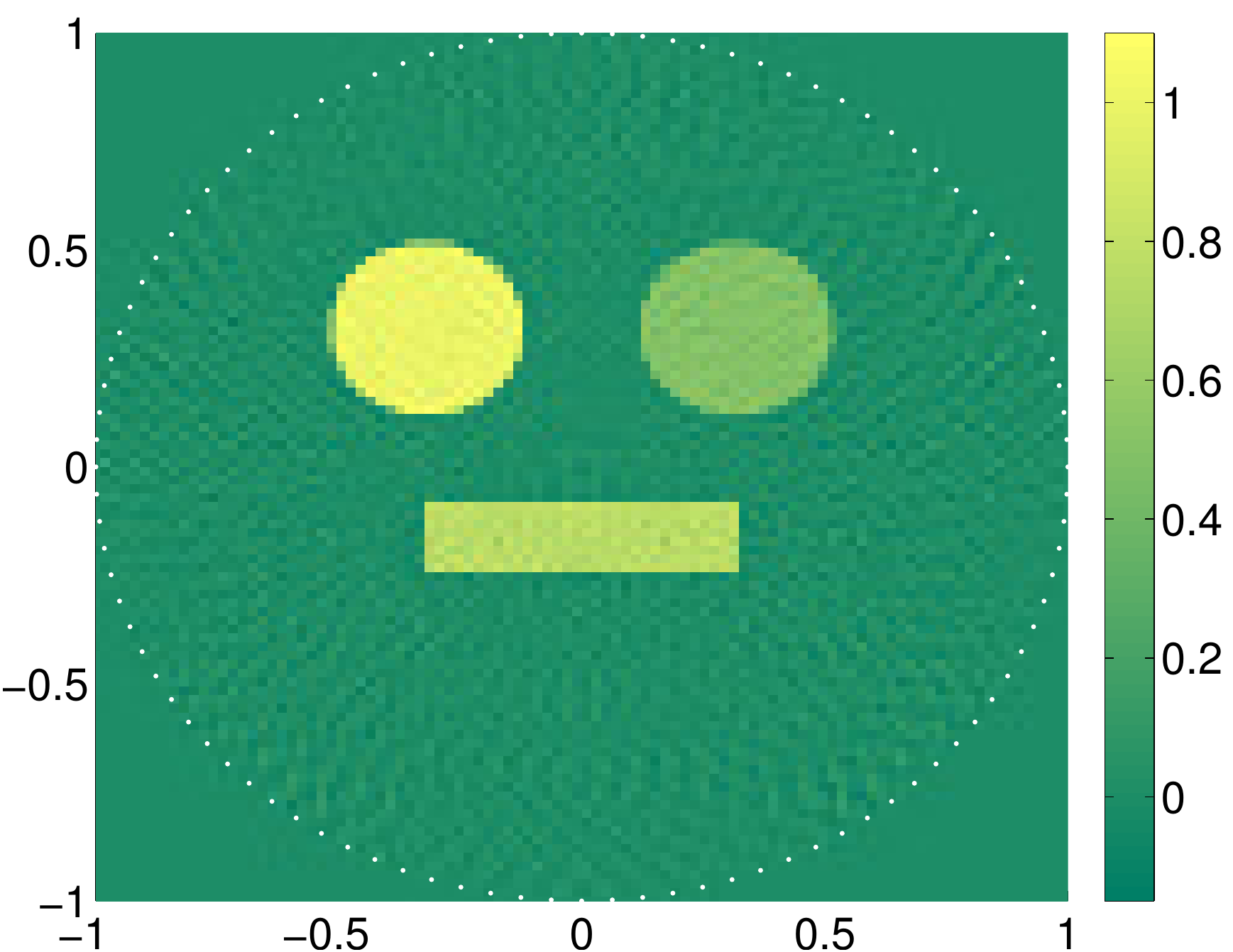}\quad
\includegraphics[width=0.45\textwidth, height=0.36\textwidth]{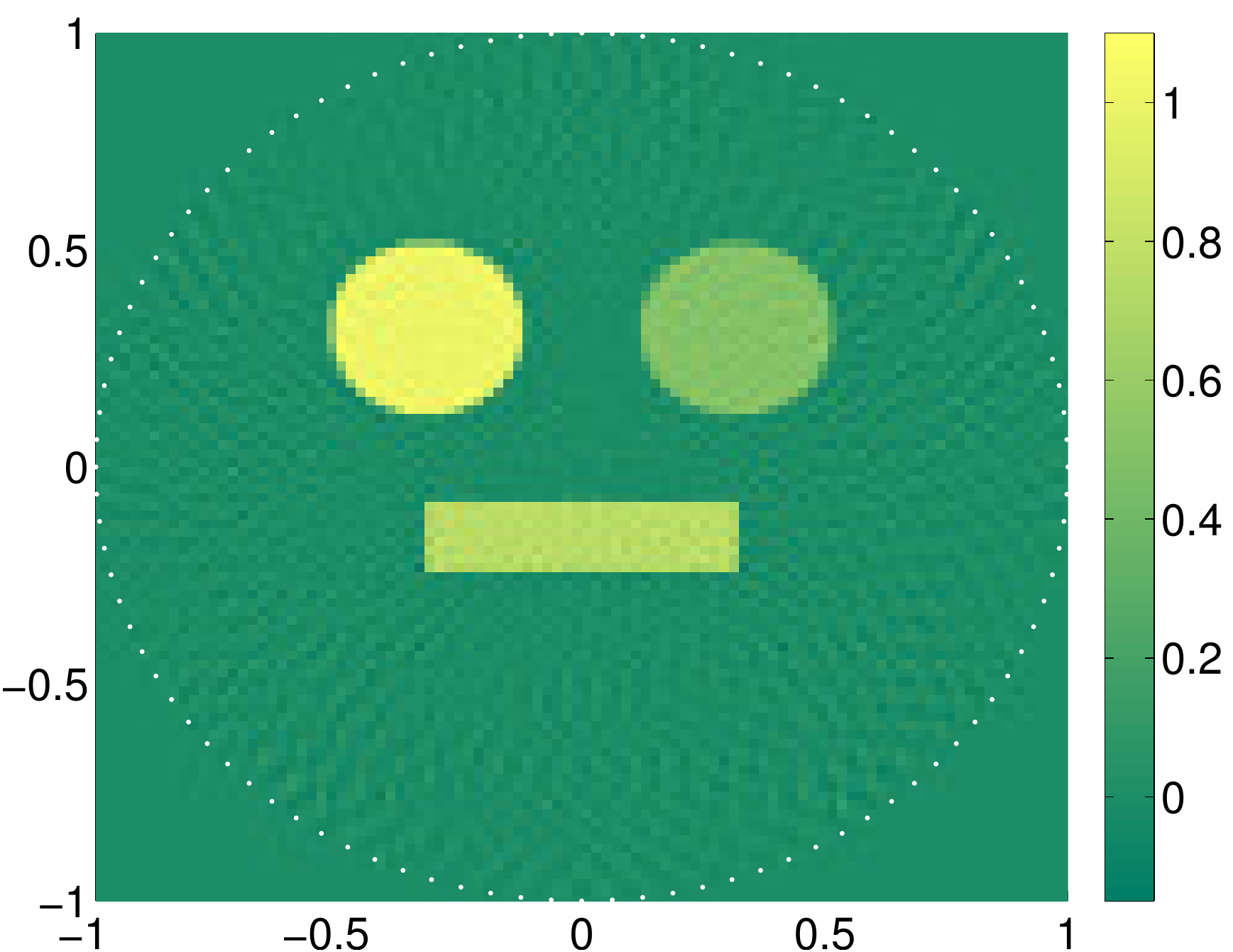}
\caption{\textsc{Comparison\label{fig:comparison} of reconstruction methods.}
Top left:   KB Galerkin approach using 40 CG iterations.
Top right: Fully discrete KB reconstruction using  40 CG iterations.
Bottom left: FBP algorithm.
Bottom right: Galerkin reconstruction using pixel basis.}
\end{figure}

\subsection{Comparison with state of the art reconstruction methods}

We compare  our  Galerkin approach using KB functions with other state of the
art approaches for PAT image reconstruction.  We used the same phantom as above
and the same wave data $\Wo f$ for all reconstruction methods. We selected $100 \times 100$ basis functions. For the KB Galerkin approach we use the generating function $\ph(\edot; 1,2,2)$
with step size parameter $s = 0.8081$ and correspondingly
$T = 0.025$.

The KB Galerkin-least squares approach is compared to the following methods:

\begin{itemize}

\item {\bfseries Discrete-discrete KB imaging model \cite{wang2014discrete}.}
We compare our method  also to the DD (discrete-discrete)
image reconstruction approach  using KB functions proposed in \cite{wang2014discrete}. There the same basis functions  for approximating the unknown function are used,
$f_N  = \sum_{k\in \La_N} c_{N,k} \ph_N^k$.
Opposed to our Galerkin approach, for  recovering  the coefficients
in  the basis expansion one forces
$\Wo f_N$ to exactly interpolate the discrete data values
$g (x_i,t_j)$. This  is equivalently
characterized  as the minimizer of
following discrete data least squares functional over $\X_N$,
\begin{equation}\label{eq:leastfully}
    \frac{1}{2} \norm{ \Bo_N  c_N - g_N}^2
	\to \min_{c_N}
\end{equation}
where  $\Bo_N := (\Wo\varphi_N^k(x_i,t_j))_{i,k}$
and $g_N \coloneqq (g (x_i,t_j))_{i,j}$.
Note that in~\cite{wang2014discrete} it has been proposed
to add an additional regularization term to~\eqref{eq:leastfully},
which we do not consider here.

\item {\bfseries Filtered backprojection (FBP) algorithm.}
For the  filtered  backprojection algorithm we implemented the explicit inversion formula
\begin{multline} \label{eq:inv}
	f(x) =\frac{2}{R}   \left(  \Wo^*t\Wo f \right)(x)
	\\=
	-\frac{1}{\pi}\int_{\partial D_1}\int_{|x-p|}^{\infty}\frac{\partial_t  (t\Wo f(p,t))}{\sqrt{t^2-|x-p|^2}}\rmd t \rmd s(p) \quad \text{ for all } x \in B_R(0) \,.
\end{multline}
The inversion formula has been derived in \cite{FinPatRak04}  for odd spatial dimension  and in  \cite{FinHalRak07} for even dimension. The inversion formula \eqref{eq:inv} can be
efficiently implemented in the form of a  filtered backprojection algorithm requiring $\mathcal{O} (N^3)$ floating operations, where $N \times N$ is the number of reconstruction points, see~\cite{BurBauGruHalPal07,FinHalRak07}.
For a fair comparison, the number of reconstruction points in the filtered backprojection algorithm  is taken equal to the
number of  basis functions in the KB Galerkin approach.

\item {\bfseries Galerkin reconstruction using the pixel basis.}

Here reconstruction space is generated by
$100 \times 100$ basis  functions given by piecewise constant functions on a square of length $2/100$ (see Section~\ref{sec:ex:pixel}). Since the pixel basis forms an orthonormal system it holds $\Ao_N = \Io_N$. The right hand side of the matrix equation is  computed as described in  Section~\ref{sec:galerkinshift}.
\end{itemize}

\begin{figure}[thb!]
\centering
\includegraphics[width=0.45\textwidth, height=0.36\textwidth]{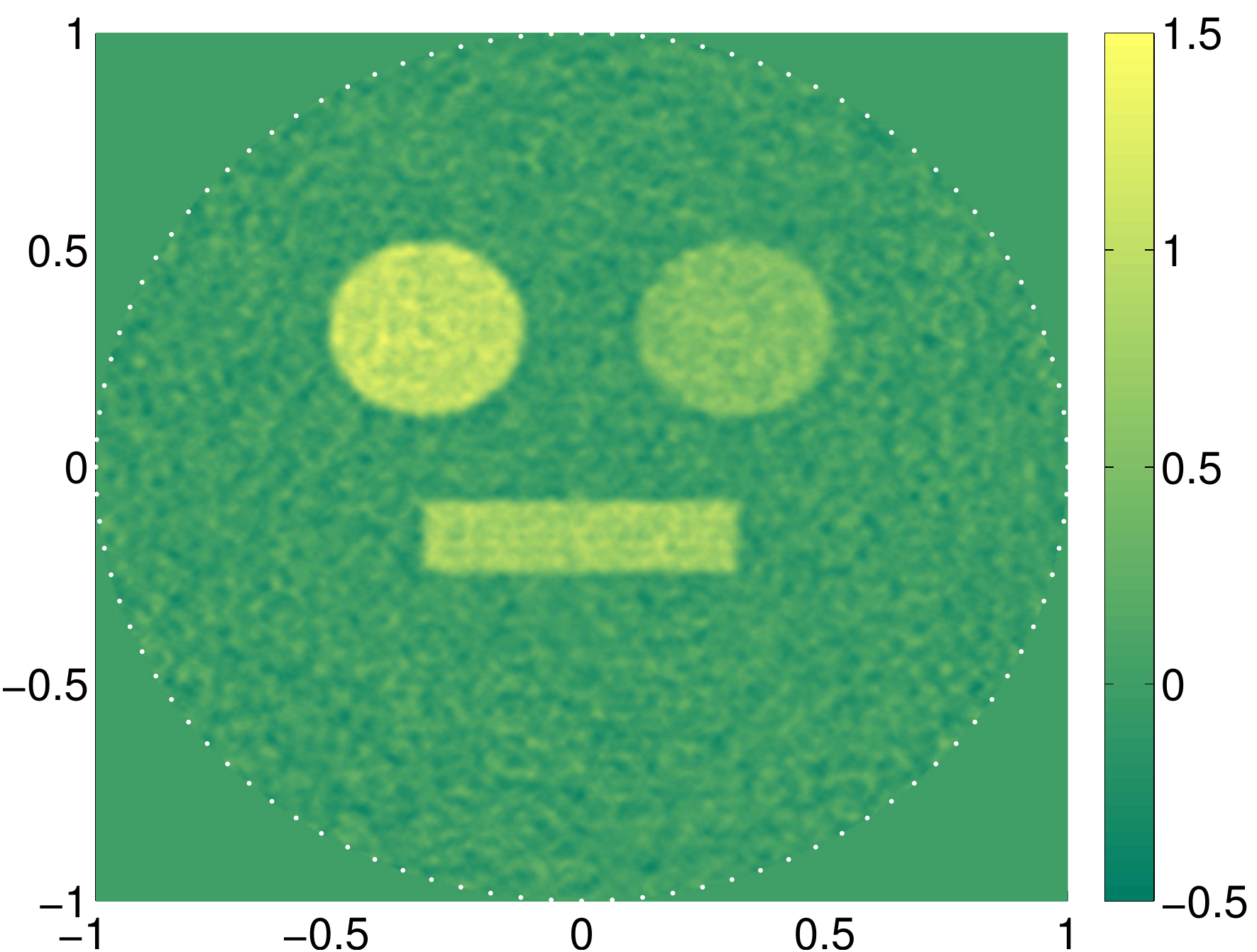}\quad
\includegraphics[width=0.45\textwidth, height=0.36\textwidth]{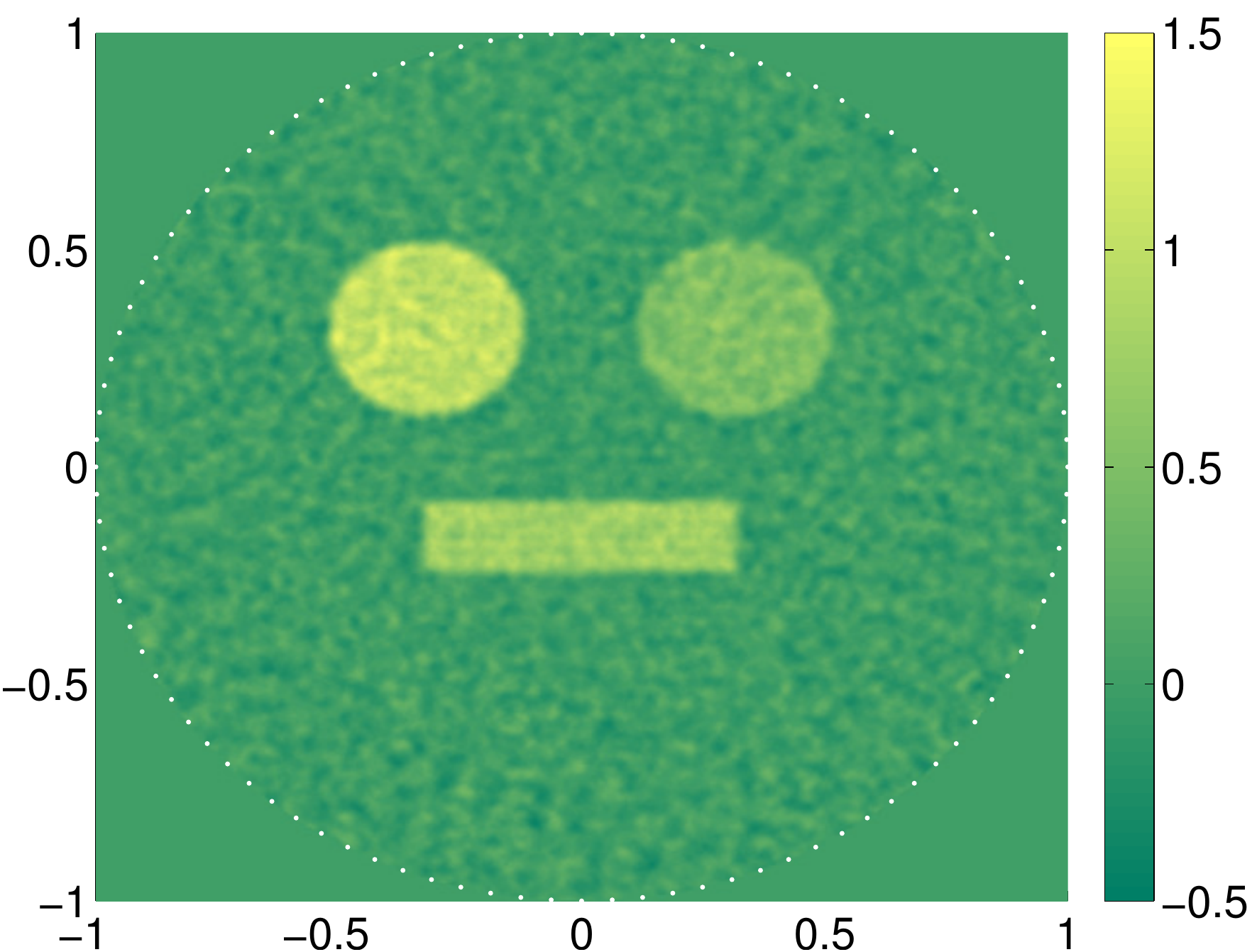}\\[0.5em]
\includegraphics[width=0.45\textwidth, height=0.36\textwidth]{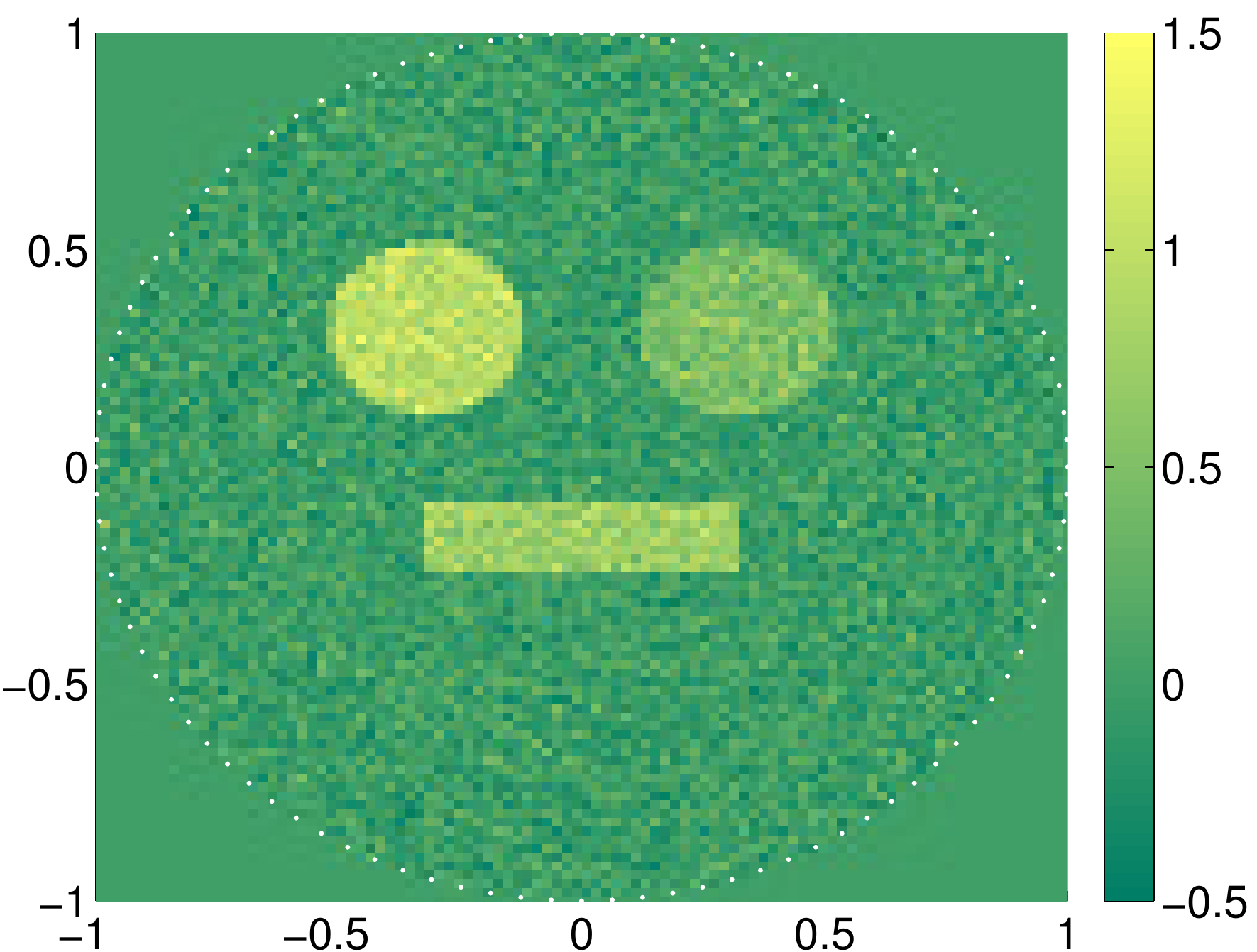}\quad
\includegraphics[width=0.45\textwidth, height=0.36\textwidth]{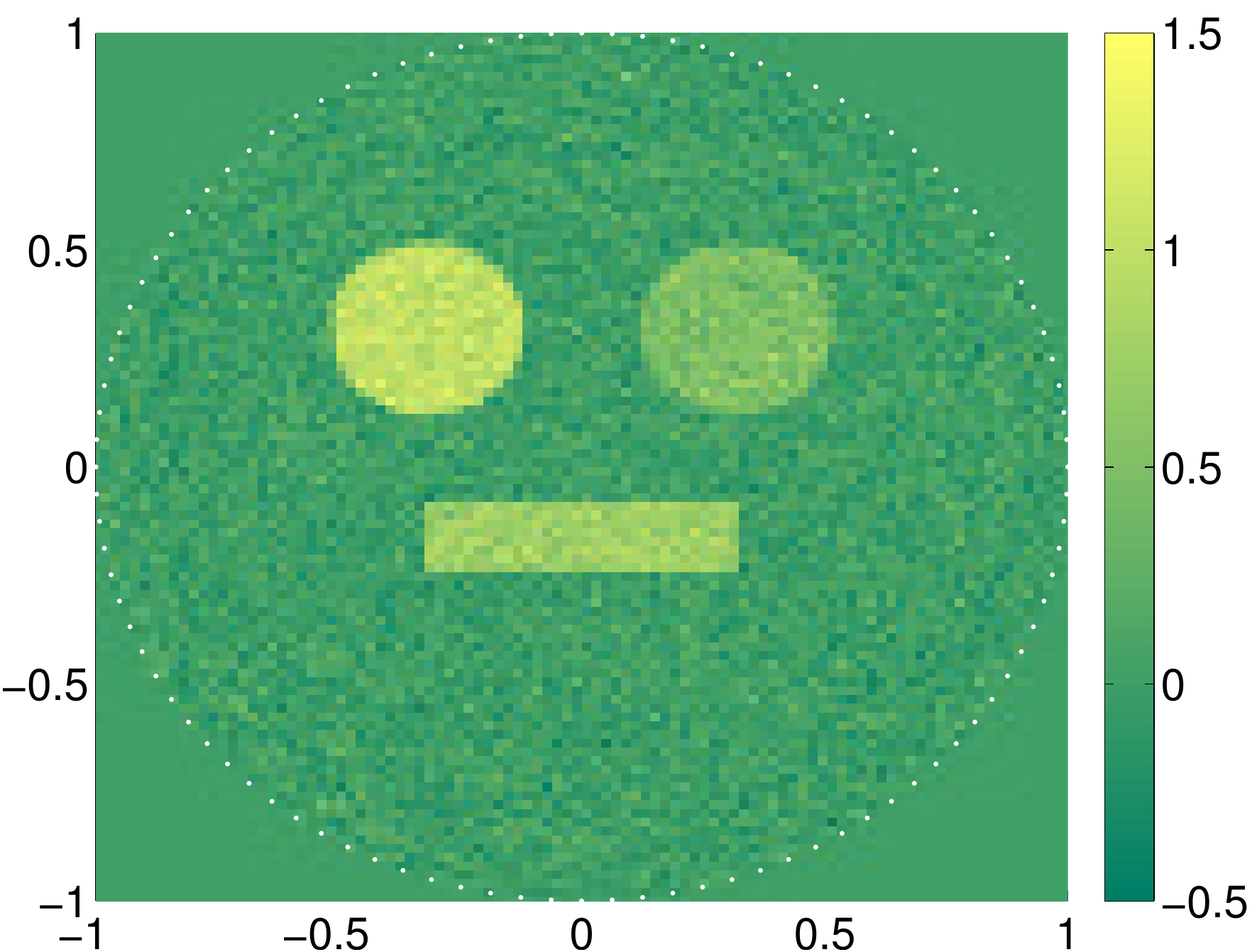}
\caption{\textsc{Comparison\label{fig:comparisonnoise} of reconstruction methods for data with  5\% noise}.
Top left:   KB Galerkin approach using 40 CG iterations.
Top right: Fully discrete KB reconstruction using  40 CG iterations.
Bottom left: FBP algorithm.
Bottom right: Galerkin reconstruction using pixel basis.}
\end{figure}

The minimizer of the  optimization problem \eqref{eq:leastfully}  is given as the solution of the normal equation
 $\Bo_N^\trans \Bo_N c =\Bo_N^\trans g_N$.
The matrix $\Bo_N^\trans \Bo_N$ is less structured and less sparse
 than our Galerkin matrix  $\Ao_N$. We observed that the direct solver in
 \MATLAB was much slower than for the Galerkin method (more than a minute compared to a fraction of a second) and therefore we decided to use iterative methods for  its solution. In particular we found the CG algorithm to perform good, which has been used
 for the results shown below. For better comparison we also  computed the KB Galerkin solution using  the CG method.
Iteratively addressing the arising equations has the advantage
that  they  are applicable  for three dimension image  reconstruction as well.

%

In Figure \ref{fig:comparison} we show reconstruction results with the above methods applied to the simulated data obtained on a standard desktop PC.    Computing the right   hand side in the Galerkin equation   took about $1.95$ seconds for the KB functions and 2.04 for the pixel basis. The solution of
the KB Galerkin equation took $0.31$ seconds with the  direct \MATLAB solver and   $0.11$ seconds using $40$ steps of the CG equation.  The solution of the  discrete equation took  about $76.17$ seconds with the  direct \MATLAB solver and   $5.01$ seconds using $40$ steps of the CG equation.   The used filtered
 backprojection algorithm took about $0.08$ seconds. One observes that
 computing the right hand side is currently the most time consuming part in
 the Galerkin approach.  Since we have to compute $N^2$ inner products and each inner product  consist of a sum over $\Ndet \Nt$ components, the numerical effort   of that step is  $\mathcal O(N^4)$
 if we take $\Ndet =  \mathcal O(N)$ and   $\Nt  =  \mathcal O(N)$.   By exploiting the  special
 structure of the  basis  functions and the wave operator we believe that
 it might be possible to derive  $\mathcal O(N^3)$ algorithm
 for evaluating the  right hand side. In such a situation   we would reach the computational  performance of the FPB algorithm with more flexibility and a potentially better accuracy.
 Further note  that  the matrix $\Bo_N^\trans \Bo_N$ in the DD approach is not sparse
 which explains why the CG method for the Galerkin approach is faster than the CG method for the DD approach.  In three spatial dimension, both the DD approach  (see \cite{wang2014discrete})
 and the Galerkin approach  yields to a sparse system matrix and therefore both have
 similar and good numerical efficiency in this case.

In order  to investigate the stability of the
above algorithms with respect to  noise we repeated the above computation after Gaussian white noise with variance equal  to  $5\%$ of the $L^2$-norm of the data. The results are shown in Figure~\ref{fig:comparisonnoise}.  Table~\ref{tab:comparison} summarizes the $L^2$-reconstruction
for different noise level and different reconstruction errors.    We see that the methods using
the KB functions  perform best in terms of the $L^2$-reconstruction error.  Note that the early
stopping of the CG methods has a regularization effect. This partially explains the  smaller reconstruction
error  of the method using the CG iteration. We emphasize  that we did not  select the number of iterations to minimize the reconstruction error. The KB Galerkin using the  direct solver in \MATLAB also gives
quite small error, which indicates that early stopping is not a very important issue in terms of
the stability. Note that for noisy data all results can be improved by incorporating
regularization (see, for example,  \cite{Hal11b} for the FBP algorithm and  \cite{wang2014discrete}
for the DD approach). 

\begin{table}[thb!]
\centering
\begin{tabular}{c |c c  c c c }
\toprule
 noise  ($\%$)      & Galerkin &  Galerkin (CG) &   DD approach  (CG) & FBP   & Pixel  \\
\midrule
$0$                      & 0.0323  & 0.0306                &   0.0314   &  0.0347           &  0.0283 \\
\midrule
$2.5$                   & 0.0830  & 0.0748                &   0.0783   &  0.2064            &  0.1249 \\
\midrule
$5$                      & 0.1411  &  0.1272               &   0.1140   &  0.3897            &  0.2092  \\
\bottomrule
\end{tabular}
\caption{\textsc{Relative\label{tab:comparison} $L^2$-reconstruction errors for $N = 100$, $s=0.8081$
using different reconstruction methods and different noise levels.} The Galerkin, the Galerkin (CG) and
DD-approach (CG) we use the the KB basis function $\ph(\edot; 1,2,2)$. For the methods using the CG algorithm
40 iterative steps  have been performed.}
\end{table}

\section{Conclusion and outlook}
\label{sec:conclusion}

In this paper we studied  (least-squares) Galerkin methods
for photoacoustic tomography with spherical geometry (and arbitrary dimension). We implemented our  Galerkin approach  for two spatial dimensional and presented  numerical results demonstrating
that yields accurate results. The considered approach yields  to solution  of  the Galerkin equation $ \Ao_N c_N =  b_N$,
where the system matrix $\Ao_N$ has size $N^2 \times N^2$ with $N^2$ denoting the number of basis elements.
For a general reconstruction space, the  system matrix to be
computed and stored is dense and unstructured. In this paper we showed that by using the isometry property of~\cite{FinHalRak07,FinPatRak04} in combination with  translation invariant reconstruction spaces, the system matrix is
sparse and has simple  structure. This can be used to easily set up the
Galerkin equation and  efficiently solve the Galerkin equation.
This is in contrast to  existing  model based approaches for two-dimensional PAT,
that do not  yield to a sparse  system matrix and numerical solvers for the arising equation (such as the CG algorithm) are numerically more expensive.

There are several possible interesting extensions and modifications
of our image reconstruction approach.  One intended line of research is the extension of  our algorithm to three spatial  dimension.
For that purpose we believe  that it is most promising to use iterative methods (such as the  CG algorithm) for solving the Galerkin equation. One advantage in this case is that the system matrix is not required
to be explicitly stored. For that purpose we will further derive more efficient
ways how to evaluate the right hand side in the Galerkin equation which is, at least for the presented algorithm in  two spatial dimensions,  the most time consuming part.
Another  practically  important extension of our framework is to incorporate
finite detector size, finite bandwidth of the  detection system and allowing
incomplete data.   In such  cases it will be necessary to include additional
regularization  to stabilize the reconstruction process.
We intend to apply our algorithm to experimental data and to study the
optimal parameter choices in such a situation. Finally it would be
interesting to extend our approach to more general measurement surfaces.

\subsection*{Acknowledgement} The authors thank the reviewers for careful reading and helpful comments on the manuscript.


\begin{thebibliography}{10}

\bibitem{AgrKuc07}
{\sc M.~Agranovsky and P.~Kuchment}, {\em Uniqueness of reconstruction and an
  inversion procedure for thermoacoustic and photoacoustic tomography with
  variable sound speed}, Inverse Probl., 23 (2007), pp.~2089--2102.

\bibitem{ansorg2013summability}
{\sc M.~Ansorg, F.~Filbir, W.~R. Madych, and R.~Seyfried}, {\em Summability
  kernels for circular and spherical mean data}, Inverse Probl., 29 (2013),
  p.~015002.

\bibitem{arridge2016adjoint}
{\sc S.~R. Arridge, M.~M. Betcke, B.~T. Cox, F.~Lucka, and B.~E. Treeby}, {\em
  On the adjoint operator in photoacoustic tomography}, Inverse Probl., 32
  (2016), p.~115012 (19pp).

\bibitem{beard2011biomedical}
{\sc P.~Beard}, {\em Biomedical photoacoustic imaging}, Interface focus, 1
  (2011), pp.~602--631.

\bibitem{belhachmi2016direct}
{\sc Z.~Belhachmi, T.~Glatz, and O.~Scherzer}, {\em A direct method for
  photoacoustic tomography with inhomogeneous sound speed}, Inverse Probl., 32
  (2016), p.~045005.

\bibitem{blu99approximation}
{\sc T.~Blu and M.~Unser}, {\em Approximation error for quasi-interpolators and
  (multi-)wavelet expansions}, Appl. Comput. Harmon. Anal., 6 (1999),
  pp.~219--251.

\bibitem{BurBauGruHalPal07}
{\sc P.~Burgholzer, J.~Bauer-Marschallinger, H.~Gr{\"u}n, M.~Haltmeier, and
  G.~Paltauf}, {\em Temporal back-projection algorithms for photoacoustic
  tomography with integrating line detectors}, Inverse Probl., 23 (2007),
  pp.~S65--S80.

\bibitem{burgholzer2007exact}
{\sc P.~Burgholzer, G.~J. Matt, M.~Haltmeier, and G.~Paltauf}, {\em Exact and
  approximate imaging methods for photoacoustic tomography using an arbitrary
  detection surface}, Phys. Rev. E, 75 (2007), p.~046706.

\bibitem{DeaBueNtzRaz12}
{\sc X.~L. Dean-Ben, A.~Buehler, V.~Ntziachristos, and D.~Razansky}, {\em
  Accurate model-based reconstruction algorithm for three-dimensional
  optoacoustic tomography}, IEEE Trans. Med. Imag., 31 (2012), pp.~1922--1928.

\bibitem{diebold1991photoacoustic}
{\sc G.~J. Diebold, T.~Sun, and M.~I. Khan}, {\em Photoacoustic monopole
  radiation in one, two, and three dimensions}, Phys. Rev. Lett., 67 (1991),
  p.~3384.

\bibitem{FinHalRak07}
{\sc D.~Finch, M.~Haltmeier, and Rakesh}, {\em Inversion of spherical means and
  the wave equation in even dimensions}, SIAM J. Appl. Math., 68 (2007),
  pp.~392--412.

\bibitem{FinPatRak04}
{\sc D.~Finch, S.~K. Patch, and Rakesh}, {\em Determining a function from its
  mean values over a family of spheres}, SIAM J. Math. Anal., 35 (2004),
  pp.~1213--1240.

\bibitem{gordon1970algebraic}
{\sc R.~Gordon, R.~Bender, and G.~T. Herman}, {\em Algebraic reconstruction
  techniques {(ART)} for three-dimensional electron microscopy and x-ray
  photography}, Journal Theor. Biol., 29 (1970), pp.~471--481.

\bibitem{Hal11b}
{\sc M.~Haltmeier}, {\em A mollification approach for inverting the spherical
  mean {R}adon transform}, SIAM J. Appl. Math., 71 (2011), pp.~1637--1652.

\bibitem{haltmeier13inversion}
{\sc M.~Haltmeier}, {\em Inversion of circular means and the wave equation on
  convex planar domains}, Comput. Math. Appl., 65 (2013), pp.~1025--1036.

\bibitem{haltmeier14universal}
{\sc M.~Haltmeier}, {\em Universal inversion formulas for recovering a function
  from spherical means}, SIAM J. Math. Anal., 46 (2014), pp.~214--232.

\bibitem{haltmeier2016iterative}
{\sc M.~Haltmeier and L.~V. Nguyen}, {\em Iterative methods for photoacoustic
  tomography with variable sound speed}.
\newblock arXiv:1611.07563, 2016.

\bibitem{HalPer15a}
{\sc M.~Haltmeier and S.~{Pereverzyev Jr.}}, {\em Recovering a function from
  circular means or wave data on the boundary of parabolic domains}, SIAM J.
  Imaging Sci., 8 (2015), pp.~592--610.

\bibitem{HalPer15b}
{\sc M.~Haltmeier and S.~{Pereverzyev Jr.}}, {\em The universal back-projection
  formula for spherical means and the wave equation on certain quadric
  hypersurfaces}, J. Math. Anal. Appl., 429 (2015), pp.~366--382.

\bibitem{HalSchBurNusPal07}
{\sc M.~Haltmeier, O.~Scherzer, P.~Burgholzer, R.~Nuster, and G.~Paltauf}, {\em
  Thermo\-acoustic tomography and the circular {R}adon transform: exact
  inversion formula}, Math. Mod. Meth. Appl. Sci., 17 (2007), pp.~635--655.

\bibitem{HalSchuSch05}
{\sc M.~Haltmeier, T.~Schuster, and O.~Scherzer}, {\em Filtered backprojection
  for thermoacoustic computed tomography in spherical geometry}, Math. Meth.
  Appl. Sci., 28 (2005), pp.~1919--1937.

\bibitem{haltmeier2010spatial}
{\sc M.~Haltmeier and G.~Zangerl}, {\em Spatial resolution in photoacoustic
  tomography: Effects of detector size and detector bandwidth}, Inverse Probl.,
  26 (2010), p.~125002.

\bibitem{herman2015basis}
{\sc G.~T. Herman}, {\em Basis functions in image reconstruction from
  projections: A tutorial introduction}, Sens. and Imaging, 16 (2015),
  pp.~1--21.

\bibitem{Hristova2008}
{\sc Y.~Hristova, P.~Kuchment, and L.~Nguyen}, {\em Reconstruction and time
  reversal in thermoacoustic tomography in acoustically homogeneous and
  inhomogeneous media}, Inverse Probl., 24 (2008), p.~055006 (25pp).

\bibitem{kak2001principles}
{\sc A.~C. Kak and M.~Slaney}, {\em Principles of Computerized Tomographic
  Imaging}, vol.~33 of Classics in Applied Mathematics, Society for Industrial
  and Applied Mathematics (SIAM), Philadelphia, PA, 2001.

\bibitem{Kre99}
{\sc R.~Kress}, {\em Linear Integral Equations}, Springer Verlag, Berlin, 1999.
\newblock second edition.

\bibitem{KruKisReiKruMil03}
{\sc R.~A. Kruger, W.~L. Kiser, D.~R. Reinecke, G.~A. Kruger, and K.~D.
  Miller}, {\em Thermoacoustic molecular imaging of small animals}, Mol.
  Imaging, 2 (2003), pp.~113--123.

\bibitem{Kuc14}
{\sc P.~Kuchment}, {\em The {R}adon transform and medical imaging}, SIAM,
  Philadelphia, 2014.

\bibitem{kuchment2011mathematics}
{\sc P.~Kuchment and L.~Kunyansky}, {\em Mathematics of photoacoustic and
  thermoacoustic tomography}, in Handbook of Mathematical Methods in Imaging,
  Springer, 2011, pp.~817--865.

\bibitem{Kun07a}
{\sc L.~A. Kunyansky}, {\em Explicit inversion formulae for the spherical mean
  {R}adon transform}, Inverse Probl., 23 (2007), pp.~373--383.

\bibitem{Kun07b}
{\sc L.~A. Kunyansky}, {\em A series solution and a fast algorithm for the
  inversion of the spherical mean {R}adon transform}, Inverse Probl., 23
  (2007), pp.~S11--S20.

\bibitem{kunyansky2015inversion}
{\sc L.~A. Kunyansky}, {\em Inversion of the spherical means transform in
  corner-like domains by reduction to the classical {R}adon transform}, Inverse
  Probl., 31 (2015).

\bibitem{lewitt1990multidimensional}
{\sc R.~M. Lewitt}, {\em Multidimensional digital image representations using
  generalized kaiser--bessel window functions}, J. Opt. Soc. Am. A, 7 (1990),
  pp.~1834--1846.

\bibitem{lewitt1992alternatives}
{\sc R.~M. Lewitt}, {\em Alternatives to voxels for image representation in
  iterative reconstruction algorithms}, Phys. Med. Biol., 37 (1992), p.~705.

\bibitem{Lou96}
{\sc A.~K. Louis}, {\em Approximate inverse for linear and some nonlinear
  problems}, Inverse Probl., 12 (1996), pp.~175--190.

\bibitem{LouMaa90}
{\sc A.~K. Louis and P.~Maass}, {\em A mollifier method for linear operator
  equations of the first kind}, Inverse Probl., 6 (1990), pp.~427--440.

\bibitem{LouSchu96}
{\sc A.~K. Louis and T.~Schuster}, {\em A novel filter design technique in 2{D}
  computerized tomography}, Inverse Probl., 12 (1996), pp.~685--696.

\bibitem{Mal09}
{\sc S.~Mallat}, {\em A wavelet tour of signal processing: The sparse way},
  Elsevier/Academic Press, Amsterdam, third~ed., 2009.

\bibitem{matej1996practical}
{\sc S.~Matej and R.~M. Lewitt}, {\em Practical considerations for 3-d image
  reconstruction using spherically symmetric volume elements}, IEEE Trans. Med.
  Imag., 15 (1996), pp.~68--78.

\bibitem{Nat01}
{\sc F.~Natterer}, {\em The Mathematics of Computerized Tomography}, vol.~32 of
  Classics in Applied Mathematics, SIAM, Philadelphia, 2001.

\bibitem{natterer2012photo}
{\sc F.~Natterer}, {\em Photo-acoustic inversion in convex domains}, Inverse
  Probl. Imaging, 6 (2012), pp.~315--320.

\bibitem{nguyen2009family}
{\sc L.~V. Nguyen}, {\em A family of inversion formulas for thermoacoustic
  tomography}, Inverse Probl., 3 (2009), pp.~649--675.

\bibitem{nguyen2016dissipative}
{\sc L.~V. Nguyen and L.~A. Kunyansky}, {\em A dissipative time reversal
  technique for photoacoustic tomography in a cavity}, SIAM J. Imaging Sci., 9
  (2016), pp.~748--769.

\bibitem{nilchian2015optimized}
{\sc M.~Nilchian, J.~P. Ward, C.~Vonesch, and M.~Unser}, {\em Optimized
  kaiser--bessel window functions for computed tomography}, {IEEE} Trans. Image
  Process., 24 (2015), pp.~3826--3833.

\bibitem{ntziachristos2005looking}
{\sc V.~Ntziachristos, J.~Ripoll, L.~V. Wang, and R.~Weissleder}, {\em Looking
  and listening to light: the evolution of whole-body photonic imaging}, Nat.
  Biotechnol., 23 (2005), pp.~313--320.

\bibitem{palamodov2012uniform}
{\sc V.~P. Palamodov}, {\em A uniform reconstruction formula in integral
  geometry}, Inverse Probl., 28 (2012), p.~065014.

\bibitem{PalNusHalBur07b}
{\sc G.~Paltauf, R.~Nuster, M.~Haltmeier, and P.~Burgholzer}, {\em Experimental
  evaluation of reconstruction algorithms for limited view photoacoustic
  tomography with line detectors}, Inverse Probl., 23 (2007), pp.~S81--S94.

\bibitem{PalViaPraJac02}
{\sc G.~Paltauf, J.~A. Viator, S.~A. Prahl, and S.~L. Jacques}, {\em Iterative
  reconstruction algorithm for optoacoustic imaging}, J. Opt. Soc. Am., 112
  (2002), pp.~1536--1544.

\bibitem{rieder2000approximate}
{\sc A.~Rieder and T.~Schuster}, {\em The approximate inverse in action with an
  application to computerized tomography}, SIAM J. Num. Anal., 37 (2000),
  pp.~1909--1929.

\bibitem{rieder2004approximate}
{\sc A.~Rieder and T.~Schuster}, {\em The approximate inverse in action {III}:
  {3D}-{D}oppler tomography}, Num. Math., 97 (2004), pp.~353--378.

\bibitem{RoiEtAl14}
{\sc H.~Roitner, M.~Haltmeier, R.~Nuster, D.~P. O'Leary, T.~Berer, G.~Paltauf,
  H.~Gr{\"u}n, and P.~Burgholzer}, {\em Deblurring algorithms accounting for
  the finite detector size in photoacoustic tomography}, J. Biomed. Opt., 19
  (2014), p.~056011.

\bibitem{RosNtzRaz13}
{\sc A.~Rosenthal, V.~Ntziachristos, and D.~Razansky}, {\em Acoustic inversion
  in optoacoustic tomography: A review}, Curr. Med. Imaging Rev., 9 (2013),
  p.~318.

\bibitem{salman14inversion}
{\sc Y.~Salman}, {\em An inversion formula for the spherical mean transform
  with data on an ellipsoid in two and three dimensions}, J. Math. Anal. Appl.,
  420 (2014), pp.~612--620.

\bibitem{Treeby10}
{\sc B.~E. Treeby and B.~T. Cox}, {\em k-wave: Matlab toolbox for the
  simulation and reconstruction of photoacoustic wave-fields}, J. Biomed. Opt.,
  15 (2010), p.~021314.

\bibitem{wang2011imaging}
{\sc K.~Wang, S.~A. Ermilov, R.~Su, H.~Brecht, A.~A. Oraevsky, and M.~A.
  Anastasio}, {\em An imaging model incorporating ultrasonic transducer
  properties for three-dimensional optoacoustic tomography}, IEEE Trans. Med.
  Imag., 30 (2011), pp.~203--214.

\bibitem{wang2014discrete}
{\sc K.~Wang, R.~W. Schoonover, R.~Su, A.~Oraevsky, and M.~A. Anastasio}, {\em
  Discrete imaging models for three-dimensional optoacoustic tomography using
  radially symmetric expansion functions}, IEEE Trans. Med. Imag., 33 (2014),
  pp.~1180--1193.

\bibitem{wang2012investigation}
{\sc K.~Wang, R.~Su, A.~A. Oraevsky, and M.~A. Anastasio}, {\em Investigation
  of iterative image reconstruction in three-dimensional optoacoustic
  tomography}, Phys. Med. Biol., 57 (2012), p.~5399.

\bibitem{wang2012photoacoustic}
{\sc L.~V. Wang and S.~Hu}, {\em Photoacoustic tomography: in vivo imaging from
  organelles to organs}, Science, 335 (2012), pp.~1458--1462.

\bibitem{xu2002timedomain}
{\sc M.~Xu and L.~V. Wang}, {\em Time-domain reconstruction for thermoacoustic
  tomography in a spherical geometry}, IEEE Trans. Med. Imag., 21 (2002),
  pp.~814--822.

\bibitem{XuWan03}
{\sc M.~Xu and L.~V. Wang}, {\em Analytic explanation of spatial resolution
  related to bandwidth and detector aperture size in thermoacoustic or
  photoacoustic reconstruction}, Phys. Rev. E, 67 (2003), pp.~0566051--05660515
  (electronic).

\bibitem{xu2005universal}
{\sc M.~Xu and L.~V. Wang}, {\em Universal back-projection algorithm for
  photoacoustic computed tomography}, Phys. Rev. E, 71 (2005), p.~016706.

\bibitem{zhang2009effects}
{\sc J.~Zhang, M.~A. Anastasio, P.~J. La~Rivi{\`e}re, and L.~V. Wang}, {\em
  Effects of different imaging models on least-squares image reconstruction
  accuracy in photoacoustic tomography}, IEEE Trans. Med. Imag., 28 (2009),
  pp.~1781--1790.

\end{thebibliography}
\end{document}